%% file: pathing.tex
\documentclass[final,reqno,onefignum,onetabnum]{siamart171218}
\usepackage{subdepth}
\usepackage{overpic}
\usepackage{amsmath,amstext,amsbsy,amssymb,mathdots}
\usepackage{booktabs,bm}
\usepackage{mathrsfs}
\usepackage{tikz,cite}
\usetikzlibrary{intersections}

\usepackage{courier}
\usetikzlibrary{positioning}
\usetikzlibrary{shapes,arrows}
\usepackage[fleqn,tbtags]{mathtools}
\usepackage{amsfonts}
\usepackage{relsize}
\usepackage{xcolor,colortbl}
\usepackage{psfrag}
\usepackage{alltt}
\usepackage{arydshln}
\usepackage{enumerate}
\usepackage{epstopdf}
\usepackage{enumitem}
\usepackage{multirow}
\usepackage{algorithmic}
\usepackage{framed}
\usepackage{cancel}
\usepackage{tikz}
\usepackage{tkz-euclide}
\usepackage{graphicx} % <- Preamble
\usepackage[caption=false]{subfig} % <- Preamble
\usepackage{amsopn}

\input{VladimirskyCommands.tex}

\usepackage{xargs}
\usepackage{comment}
\usepackage[colorinlistoftodos,prependcaption,textsize=tiny]{todonotes}
\reversemarginpar
\newcommandx{\avtodo}[2][1=]{\todo[linecolor=red,backgroundcolor=red!25,bordercolor=red,#1]{#2}}
\newcommandx{\mgtodo}[2][1=]{\todo[linecolor=blue,backgroundcolor=blue!25,bordercolor=blue,#1]{#2}}

\colorlet{lightgray}{gray!40}

\newcolumntype{"}{@{\hskip\tabcolsep\vrule width 1pt\hskip\tabcolsep}}
\setlist[description]{font=\normalfont\space}

\newcommand{\Ep}{\mathcal{A}} % Set of E pure strategies
\newcommand{\Op}{\mathcal{X}} % Set of O pure strategies
\newcommand{\Em}{\Delta_{\mathcal{A}}} % Set of E mixed strategies
\newcommand{\Om}{\Delta_{r}} % Set of O mixed strategies
\newcommand{\bap}{\ba ( \cdot )} % Set of O mixed strategies
\newcommand{\Es}{\bm{\theta}} % name of E's mixed strategy
\newcommand{\Os}{\bm{\lambda}} % name of O's mixed strategy
\newcommand{\bomega}{\bm{\omega}} % Set of E pure strategies
\newcommand{\e}{\bm{e}} % std basis vector
\newcommand{\Res}{\bm{R}} % residual vector
\newcommand{\bJ}{\bm{\mathcal{J}}} % Set of E pure strategies
\newcommand{\A}{\mathcal{A}} % set of E strategies.
\newcommand{\K}{K} % pointwise cost function

\graphicspath{{fig/}}

\usepackage{lipsum}
\ifpdf
  \DeclareGraphicsExtensions{.eps,.pdf,.png,.jpg}
\else
  \DeclareGraphicsExtensions{.eps}
\fi

\newsiamremark{remark}{Remark}
\newsiamremark{hypothesis}{Hypothesis}
\crefname{hypothesis}{Hypothesis}{Hypotheses}
\headers{Surveillance-Evasion path planning}{M.A. Gilles and A. Vladimirsky}

\title{Evasive path planning under surveillance uncertainty\thanks{Submitted to the editors on 12/26/18
\funding{This work is supported in part by the National
Science Foundation grant DMS-1738010.  The second author's work is also supported by the Simons Foundation Fellowship.}}}
\author{Marc Aur\`ele Gilles\thanks{Center for Applied Mathematics, Cornell University, Ithaca, NY 14853
  (\email{mtg79@cornell.edu}).}
\and Alexander Vladimirsky\thanks{Department of Mathematics and Center for Applied Mathematics, Cornell University, Ithaca, NY 14853
  (\email{vladimirsky@cornell.edu}).}}

\begin{document}
\maketitle
\begin{abstract}
The classical setting of optimal control theory assumes full knowledge of the process dynamics and the costs associated with every control strategy.  The problem becomes much harder if the controller only knows a finite set of possible running cost functions, but has no way of checking which of these running costs is actually in place.
In this paper we address this challenge for a class of evasive path planning problems on a continuous domain, in which an Evader needs to reach a target while
minimizing his exposure to an enemy Observer, who is in turn selecting from a finite set of known surveillance plans.
Our key assumption is that both the evader and the observer need to commit to their (possibly probabilistic) strategies in advance and cannot immediately change their actions based on any newly discovered information about the opponent's current position. We consider two types of evader behavior: in the first one, a completely risk-averse evader seeks a trajectory minimizing his {\em worst-case} cumulative observability, and in the second, the evader is concerned with minimizing the {\em average-case} cumulative observability.
The latter version is naturally interpreted as a semi-infinite strategic game, and we provide an efficient method for approximating its Nash equilibrium.
The proposed approach draws on methods from game theory, convex optimization, optimal control, and multiobjective dynamic programming.
We illustrate our algorithm using numerical examples and discuss the computational complexity, including for the generalized version with multiple evaders.
\end{abstract}



\input{Introduction.tex}

\input{ContinuousPath.tex}

\input{MultiObjective.tex}

\input{ObserverEvader.tex}

\input{Examples.tex}

\input{Extensions.tex}

\input{Conclusion.tex}

\noindent
{\bf Acknowledgements: }
The authors would like to thank Alex Townsend and anonymous reviewers for their helpful suggestions.

\bibliographystyle{siam}
\bibliography{pathing}

\end{document}

%% file: VladimirskyCommands.tex
\newcommand*{\CopyCounter}[2]{%
  \expandafter\def\csname c@#2\endcsname{\csname c@#1\endcsname}%
  \expandafter\def\csname p@#2\endcsname{\csname p@#1\endcsname}%
  \expandafter\def\csname the#2\endcsname{\csname the#1\endcsname}}

\CopyCounter{Theorem}{ProposedProblem}
\CopyCounter{Theorem}{Proposition}
\CopyCounter{Theorem}{Property}
\CopyCounter{Theorem}{Claim}
\CopyCounter{Theorem}{Lemma}
\CopyCounter{Theorem}{Corollary}
\CopyCounter{Theorem}{Conjecture}
\CopyCounter{Theorem}{Definition}
\CopyCounter{Theorem}{Example}
\CopyCounter{Theorem}{Remark}
\CopyCounter{Theorem}{Question}
\CopyCounter{Theorem}{Condition}
\CopyCounter{Theorem}{Criterion}
\CopyCounter{Theorem}{Observation}
\theoremstyle{plain}%% needs amsthm.sty

\newcommand{\domain}{\Omega}

\newcommand{\boundary}{\partial \domain}

\newcommand{\xhat}{\bm{\hat{x}}}

\newcommand{\x}{\bm{x}}

\newcommand{\ba}{\bm{a}}

\newcommand{\Bf}{\bm{f}}

\newcommand{\w}{\bm{w}}
\newcommand{\y}{\bm{y}}

\newcommand{\J}{{\cal{J}}}

\newcommand{\R}{\mathbb{R}}

\DeclareMathOperator*{\argmin}{arg\,min}
\DeclareMathOperator*{\argmax}{arg\,max}

%% file: Introduction.tex
% !TEX root = draft.tex

\section{Introduction}

Path planning is a problem of interest for many communities: traffic engineering, autonomous driving, robotics, and military. In the classical setting, the path planner is assumed to have full information about the environment and chooses a path minimizing some undesirable quantity; e.g., time-to-target, distance traveled, fuel consumption, or threat exposure.
A particular type of continuous path planning problems is surveillance-evasion applications. In the simplest scenario, an evader (E) is choosing a path to minimize its exposure to an observer (O) whose surveillance plan is fixed and fully known to E in advance.
This formulation is conveniently treated in the framework of {\em optimal control theory}, reviewed in ~\cref{sec:ContinuousPath},
with the evader's optimal policy recovered by solving a Hamilton-Jacobi-Bellman (HJB) partial differential equation (PDE).
But the real focus of this paper is on path planning under uncertainty, where E knows the full list of different surveillance plans available to O but does not know which of them is currently in use.

The key assumption of our model is that neither E nor O can change their respective strategy in real time based on the opponent's discovered position or actions.  In practical contexts (e.g., in satellite-based surveillance), this restriction might be due to either a delayed analysis of observations or due to logistical needs of committing to a strategy in advance.
As in many other optimization under uncertainty situations, it is natural for E to treat this as an {\em adversarial} problem -- either because the prior statistics on the frequency of use for specific surveillance plans are unreliable or because O might be actively adjusting these frequencies in response to E's routing choices.

In considering each potential path to its destination, E needs to evaluate the trade-offs in observability with respect to %several
different surveillance plans.  This naturally brings us to the notion of {\em Pareto optimality} \cite{Marler} and the numerical methods developed for multi-objective optimal control problems \cite{MitchellSastry,KumarVlad,Guigue,DesillesZidani}.  As we show in~\cref{sec:Pareto_Worst_Expected}, the method introduced in \cite{KumarVlad} can be used to find the deterministic optimal policy for a completely risk-averse evader (i.e., minimizing the worst-case observability).  Unfortunately,
the computational cost of this approach grows exponentially with the number of surveillance plans available to O.
But if the goal for both players is to optimize the average-case/expected observability, we show that this can be accomplished by adopting a much more computationally affordable method from \cite{MitchellSastry}, despite its significant drawbacks for general multi-objective control problems.  Moreover, we show that, if the evader's average-case optimal strategy is deterministic, then that same strategy is also worst-case optimal.

For the rest of the paper, we concentrate on the average-case observability formulation using a semi-infinite zero-sum game \cite{tijs1976semi}
between E and O, each of them searching for the best randomized/mixed strategy -- an optimal probability distribution over that player's available deterministic/``pure'' options.
We refer to these as ``Surveillance-Evasion Games'' (SEGs), although the same terminology was previously used in the 1960s and 1970s to describe a very different class of problems, where the Evader needs to escape from the Observer's surveillance zone as quickly as possible\cite{dobbie1966solution, LewinThesis, LewinBreakwell,LewinOlsder}.
Aside from this terminological overlap, those earlier papers have little in common with our context since in them E and O operated with full information on their opponent's current state, reacted in real time, and sought optimal feedback policies recovered by solving Hamilton-Jacobi-Isaacs equations.

In classical (finite zero-sum two-player) strategic games, the Nash equilibrium is typically obtained using linear programming \cite{osborne1994course}.
But the fact that E's set of pure strategies is uncountably infinite makes this approach unusable in our SEGs.
Instead, we show how to compute the Nash equilibrium in~\cref{sec:ObserverEvader} by combining convex optimization with fast numerical methods for HJB equations.  The computational cost of the resulting method scales at most linearly with the number of surveillance plans.  We illustrate this approach on a large number of examples, with the details of our numerical implementation covered in~\cref{sec:examples}.

We note that the same ideas are also useful outside of surveillance-evasion context, whenever the path planner cannot assess the actually incurred running cost until it reaches the target.
In fact, the same PDEs and semi-infinite zero-sum games can be used to model civilians' routing choices in war zones and other dangerous environments, minimizing their exposure to bomb threats.

Our modeling approach is quite general, but to simplify the exposition we will assume that the evader is moving in a two-dimensional domain with occluding/impenetrable obstacles,  both the observability and E's speed are {\em isotropic} (i.e., independent of E's chosen direction of motion), and all O's surveillance plans are stationary (i.e., the observer is choosing among possible stationary locations).  This further simplifies the PDE aspect of our problem from a general HJB to stationary {\em Eikonal} equations, the efficient numerical methods for which are particularly well-developed in the last 25 years (e.g., \cite{SethBook2}).

In~\cref{sec:MultipleEvaders}, we generalize the problem by considering multiple evaders.  We first treat this as a two-player game between a single observer and a centralized controller of all evaders.  But we also show that the resulting set of strategies is a Nash equilibrium even from the point of view of individual/selfish evaders.
We conclude by discussing further extensions and limitations of our approach in~\cref{sec:Conclude}.

%% file: ContinuousPath.tex
% !TEX root = draft.tex

\section{Continuous path planning}
\label{sec:ContinuousPath}
The case where the observer's strategy is fixed and known can be easily handled by methods of classical optimal control theory.
The goal is to guide an evader (E) from its starting position $\x_S$ to its desired target $\x_T$ while minimizing the ``cumulative observability'' (also called ``cumulative cost") along the way through its state space represented by some compact set $\domain \subset \R^d.$
More precisely,  we will suppose that $A$ is a compact set of control values, and $\Ep$ is the set of E's admissible controls
which are measurable functions $\ba: \R \mapsto A.$  The evader's dynamics are defined by
$\y'(t) = \Bf(\y(t),\ba(t))$, with the initial state $\y(0) = \x \in \domain.$ (Even though E only cares about the optimal trajectory from $\x_S$, the method we use encodes optimal trajectories to $\x_T$ from all starting positions $\x.$)
In all of our numerical examples, we will assume that E's state is simply its position, $\Bf$ is its velocity defined on a known map $\domain$
that excludes (impenetrable, occluding) obstacles, and E is allowed to travel along $\boundary$ (including the obstacle boundaries).
Suppose $T_{\ba} = \min \{t \geq 0 \mid \y(t) = \x_T \}$ is
the travel-time-through-$\domain$ associated with this control. A pointwise observability function (also called cost function)
$\K: \domain \times A \mapsto \R$ is then defined to reflect O's surveillance capabilities for different parts of the domain,
taking into account all obstacles/occluders and E's current position and direction.
The cumulative observability is then defined by integrating $\K$ along a trajectory corresponding to $\ba(\cdot)$ with initial position $\x$
\begin{equation}\label{eq:IntegratedObservability}
\J(\x, \ba(\cdot)) = \int_0^{T_{\ba}} \K (\y(t), \ba(t)) \, dt,
\,
\end{equation}
which we will also denote as $\J(\ba( \cdot))$ when $\x$ is clear from the context.
As usual in dynamic programming, the {\it value function} is then defined
by minimizing the cumulative observability:
$u(\x) = \inf_{\ba(\cdot)} \J(\x, \ba(\cdot))$,
with the infimum taken over controls leading to $\x_T$ without leaving $\domain$ (i.e., $T_{\ba} < \infty$
and $\y(t) \in \domain, \forall t \in [0,T_{\ba}]$ along the corresponding trajectory).
Under suitable ``small-time controllability'' assumptions \cite{Bardi_book},
it is easy to show that $u$ is locally Lipschitz on $\domain$.  If it is also
smooth, a Taylor series expansion can be used to show that $u$
satisfies a static Hamilton-Jacobi-Bellman PDE:
\begin{equation}
\min\limits_{\ba \in A}
\left\{
\K(\x, \ba) + \nabla u (\x) \cdot \Bf ( \x, \ba )
\right\}
\, = \, 0,
\; \; \; \forall \x \in \domain \setminus \{ \x_T \};
\hspace*{1.3cm} u(\x_T)=0,
\label{eq:HJB_general}
\end{equation}
with the special treatment at $\boundary \setminus \{ \x_T \}$ where the minimum is taken over the subset of control values $A$ that ensure staying inside $\domain.$

Unfortunately, the value function $u$ is generically non-smooth, and there
usually are starting positions with multiple optimal trajectories  -- these
are the locations where the characteristics cross and $\nabla u$ is
undefined. Thus, a classical solution to \cref{eq:HJB_general} usually does
not exist. The theory of {\it viscosity solutions} introduced by Crandall and
Lions \cite{CranLion} overcomes this difficulty by selecting the unique weak
solution coinciding with the value function of the original control problem.
Restricting the process dynamics to $\domain$ is similarly handled by
switching to domain-constrained viscosity solutions \cite{Soner1_1986, Bardi_book}.

To simplify the exposition, we focus on {\it isotropic} problems, where the observability $\K$ and the speed of motion $f$
depend only on $\x$. In this case, we choose $A=\{ \ba \in \R^d \, \mid \, |\ba| = 1 \}$
and interpret $\ba$ as the direction of motion.  Then $\K(\x, \ba) = \K(\x)$
and $\Bf ( \x, \ba ) = f ( \x) \ba$, with $f$ encoding the speed of motion
through the point $\x$. In this case, the optimal direction is known
analytically: $\ba^* = -\nabla u / |\nabla u|$ and \cref{eq:HJB_general}
reduces to an {\it Eikonal equation}
\begin{equation}
| \nabla u (\x) | f(\x) \; = \; \K(\x),
\qquad \forall \x \in \domain \setminus \{ \x_T \};
\hspace*{2cm} u(\x_T)=0.
\label{eq:Eikonal}
\end{equation}

The characteristics of these static PDEs are precisely the optimal trajectories, which define the direction of ``information flow''.
This is quite useful once \eqref{eq:Eikonal} is discretized on a grid (e.g., substituting upwind divided differences for partial derivatives, while taking
$u=+\infty$ for all gridpoints outside of $\domain$ to enforce the state constraints).  The discretization yields a large coupled system of nonlinear equations.
Knowing the characteristic direction for every gridpoint,  one could, in principle,
re-order the equations, effectively decoupling this system. But since the PDE
is nonlinear, its characteristic directions are not known in advance.
One path\footnote{
Fast Sweeping \cite{zhao2005fast} is another popular approach for gaining efficiency in solving Eikonal equations.
We refer readers to \cite{ChacVlad1, ChacVlad2} for a review of many other ``fast'' techniques, including the hybrid marching/sweeping methods that aim to combine the best features of both approaches.
Even though our own implementation is based on Fast Marching, any of these methods can be used to solve isotropic control problems arising in subsequent sections.  Which one will be faster depends on the domain geometry and the particular pointwise observability functions.}
 to computational efficiency is to
determine those characteristic directions
simultaneously with solving the discretized system, in the spirit
of Dijkstra's classical algorithm for finding shortest paths on graphs \cite{Diks}.
Two such non-iterative methods (Tsitsiklis' Algorithm
\cite{Tsitsiklis} and Sethian's Fast Marching Method
\cite{SethFastMarcLeveSet}) are applicable to this special isotropic case.
Once~\cref{eq:Eikonal} is solved, the optimal trajectory may be recovered by finding the path orthogonal to the level curves of $u(x)$. This can be achieved numerically by the steepest descent method on~$u(x)$. An example of the solution of~\cref{eq:Eikonal} is shown in~\ref{fig:OneObserver}.

\begin{figure}[hhhh]
    \centering
\subfloat[]{\label{fig:a}\includegraphics[width=0.49\linewidth,keepaspectratio]{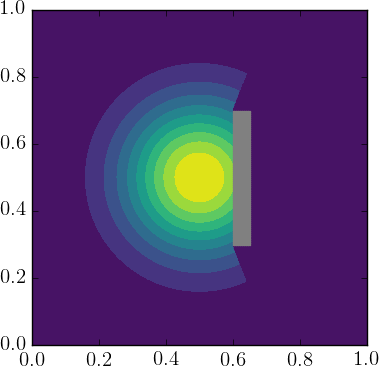}}
\subfloat[]{\label{fig:b}\includegraphics[width=0.49\linewidth,keepaspectratio]{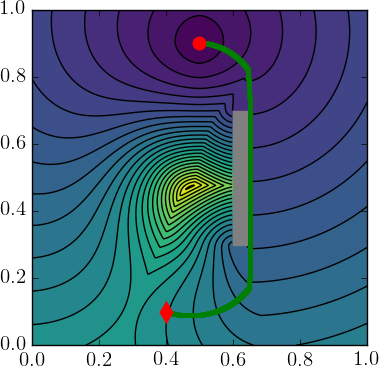}}
\begin{comment}
\begin{center}
$
\begin{array}{cc}
{\includegraphics[width=0.49\linewidth,keepaspectratio]{OneObserverK0}} &
{\includegraphics[width=0.49\linewidth,keepaspectratio]{OneObserverSol}} \\
(A)  & (B)
\end{array}
$
\end{center}
\end{comment}

\caption{ \footnotesize
    (a) The observability function $\K(\x)$ for an observer position $(0.5,0.5)$. The gray rectangle is an obstacle, which obstructs the vision of the observer. The shadow zones created by the obstacle can be computed using the solution of the Eikonal equation (see~\cref{sec:parameters}).
    (b) A contour plot of the solution of \cref{eq:Eikonal} for $f(\x)=1$ and the cost function in (a). The red diamond is the starting position, the red circle is the target position, and the green curve is the optimal trajectory, which is orthogonal to the level curves of $u(x)$
    and follows a part of the obstacle boundary.
    See~\cref{sec:examples} for additional information and parameters used.
}
\label{fig:OneObserver}
\end{figure}

%% file: MultiObjective.tex
% !TEX root = draft.tex

\section{Multiple observer locations and different notions of optimality} \label{sec:Pareto_Worst_Expected}
We now transition to the setting where the observer has a choice of multiple surveillance plans.
Assuming that the observer remains stationary, this is equivalent to choosing  its position from a fixed set of $r$ locations $\Op = \{ \hat{\x}_1, \cdots , \hat{\x}_r \}$ known to the evader.  Each location is associated with a pointwise observability function $\K_i ( \x )$ for an evader moving through $\x \in \domain$ and an observer stationed at $\hat{\x}_i.$   (Typically, $\K_i$ is a decreasing function of $|\x-\hat{\x}_i|$ when $\x$ is visible from $\hat{\x}_i$ or a small constant $\sigma>0$ if $\x$ is in a ``shadow zone''; see~\cref{sec:examples} for further details.)
This results in $r$ different definitions of the cumulative observability $\bJ = \left[ \J_1, \ldots, \J_r \right]^T$  for a particular control.  Ideally, a rational evader would prefer a path that minimizes its exposure to all possible observer locations $\hat{\x}_i$. Unfortunately, there usually does not exist a single control minimizing all $\J_i$'s simultaneously.  This naturally leads us to a notion of Pareto optimal trajectories and the methods for computing them efficiently.  We review two such methods\footnote{
Here we describe these methods in terms of exposure to different observer's positions, but both of them were introduced for much more general multi-objective control problems. In many applications it is necessary to balance completely different criteria; e.g., time vs fuel vs money vs threat, etc.
Other methods for approximating the full PF can be found in \cite{Guigue} and \cite{DesillesZidani}.}
in~\cref{sec:MultiObjective} and explain how they can be used for planning by an evader optimizing either the worst-case or average-case observability in~\cref{sec:2_approaches}.

\subsection{Multiobjective path planning}\label{sec:MultiObjective}
For a fixed starting position $\x \in \domain$, a control
$\ba(\cdot)$ is {\it dominated} by a control $\hat{\ba}(\cdot)$ if
$\J_i(\x, \hat{\ba}(\cdot)) \leq \J_i(\x, \ba(\cdot))$ for all $i$ and the
inequality is strict for at least one of them. We call $\ba(\cdot)$ {\it
Pareto optimal} if it is not dominated by any other control. In other words,
Pareto optimal controls are the ones that cannot be improved with respect to any one criterion without making them worse with respect to another.
The vector of costs associated with each Pareto optimal control defines a point in $\R^r$ and the set of all such points is
the Pareto Front (PF). In path planning applications, the PF is typically used to carefully evaluate all tradeoffs.
(E.g., what is the smallest attainable $\J_1$ given the desired upper bounds
on $\J_2, \ldots, \J_r$ ?)

Mitchell and Sastry developed a method for multiobjective path planning\cite{MitchellSastry}
based on the usual {\it scalarization} approach to multiobjective
optimization \cite{Marler}. Let $\Om = \{ \Os = (\lambda_{1}, \ldots \lambda_r)
\, \mid \, \sum_{i=1}^r \lambda_i = 1, \, \text{ and all } \lambda_i \geq 0
\}$. For each $\Os \in \Om$ one can define a new pointwise observability function
$\K^{\Os} = \sum_{i=1}^r \lambda_i \K_i$ and a new cumulative observability function $\J^{\Os} = \sum_i \J_i$. A weighted cost Eikonal equation
\begin{equation} \label{eq:WeightedCostEikonal}
| \nabla u^{\Os}(\x)| f(\x) = \K^{\Os}(\x)
\end{equation}
is then solved for a fixed $\Os$, providing a control function $\ba^{\Os} ( \cdot )$ satisfying $\ba^{\Os} ( \cdot)  \in \argmin_{\ba( \cdot ) \in \Ep } \J^{\Os}(\x_S, \ba ( \cdot ) )$.
We call such a control function $\Os$-optimal.
If $\lambda_i > 0$ for all $i$, the obtained ${\Os}$-optimal control is also guaranteed to be Pareto optimal; see ~\cref{fig:pf_scal}.
However, if at least one $\lambda_i = 0$ and multiple $\Os$-optimal strategies exist for a specific $\Os$, then some of the $\Os$-optimal strategies may not be Pareto optimal.
Such corner cases (illustrated in~\cref{fig:observability3}) might require additional pruning; alternatively, one can rule out such non-Pareto trajectories by perturbing ${\Os}$ to ensure that all components are positive.

Additional linear PDEs can be solved simultaneously to compute the individual costs $(\J_1,
\ldots \J_r)$ incurred along any $\Os$-optimal trajectory; e.g., when $f$ and all $\K_i$'s are isotropic,
the corresponding linear equations are
\begin{equation} \label{eq:associatedlinear}
\nabla v_i^{\Os} \cdot \nabla u^{\Os} = \K_i \K^{\Os} / f^2,
\end{equation}
where $v_i^{\Os}(\x)  = \J_i \left( \x, \ba^{\Os} ( \cdot ) \right).$

To approximate the PF, this procedure is repeated for a large number of  $\Os \in
\Om$.
Unfortunately, as shown in~\cref{fig:pf_scal},
scalarization-based methods
can only recover the convex portion of PF \cite{Das}.
\input{fig_pf_schematic.tex}
This is an important drawback since in many optimal control problems the non-convex
parts of PF are very common and equally important.  An alternative approach was developed in  \cite{KumarVlad} to address this limitation and produce the entire PF for all $\x \in \domain$ simultaneously.
The method is applicable for any number of observer positions, but to simplify the notation we explain it here for $r=2$ only.
We expand the state space to $\domain_e = \domain \times [0, B]$ and define
the new value function $w(\x, b) = \inf \J_1(\x, \ba(\cdot))$, with the
infimum taken over all controls satisfying $\J_2(\x, \ba(\cdot)) \leq b$.
Thus, $b$ is naturally interpreted as the current ``budget'' for the
secondary criterion. The value function is then recovered by solving an
augmented PDE
\begin{equation}
\min\limits_{\ba \in A}
\left\{
\K_1 (\x, \ba) + \nabla_{\x} w \cdot \Bf ( \x, \ba )
\, - \,
\K_2(\x, \ba) \frac{\partial w}{\partial b}
\right\}
\, = \, 0.
\label{eq:multiobject}
\end{equation}
The method in \cite{KumarVlad} uses a first-order accurate semi-Lagrangian discretization
\cite{Falcone_book} to compute the discontinuous viscosity solution of
\eqref{eq:multiobject} for a range of problems in multi-criterion path planning.
The method was later generalized to treat constraints on reset-renewable resources \cite{BudgetReset}.
The same approach was also adapted to Probabilistic RoadMap graphs and field-tested on
robotic platforms at the United Technologies Research Center \cite{UTRC_2}.

Aside from approximating the entire PF, the key computational advantage is the {\em explicit causality}:
since $\K_2$ is positive, all characteristics are monotone in $b$ and methods similar to the explicit
``forward marching'' in $b$-direction are applicable.  (I.e., the system of discretized equations
is trivially de-coupled.)  Of course, the main
drawback of the above idea is the higher dimensionality of $\domain_e$. For
$r$ observer locations, the scalarization approach \cite{MitchellSastry} requires
solving $(r+1)$ PDEs on $\domain \subset \R^d$, but the parameter space
$\Lambda_r$ is $(r-1)$-dimensional. In contrast, with $w(\x, b)$ there are no parameters,
but the computational domain is $(d+r-1)$-dimensional. Several techniques for restricting the computations to a relevant part of $\domain_e$
were developed in \cite{KumarVlad},
but the computational cost and memory requirements are still significantly higher than for any (single)
HJB-solve in $\domain$.

\subsection{Different notions of adversarial optimality}\label{sec:2_approaches}
The Pareto Front allows us to answer one version of the surveillance-evasion problem: if the evader is completely risk-averse, he may choose the {\em worst-case optimal} strategy. That is, E will pick a control $\ba_W ( \cdot )$ that minimizes the observability from its ``worst'' observer position $\hat{\x}_i$:
\[
\max_{\hat{\x}_i \in \Op} \J_i (\ba_W ( \cdot ))  \; \leq \;  \max_{\hat{\x}_i \in \Op} \J_i (\ba ( \cdot )), \qquad \forall \ba(\cdot) \in \Ep.
\]
This corresponds to the version of the problem where E is forced to ``go first", with O selecting the maximizing $\hat{\x}_i \in \Op$ in response.
The following result shows that the
intersection of Pareto Front with
the ``central ray" (i.e., the line where $\J_1 = \J_2 \dots = \J_r$) yields the worst-case optimal strategy for E:
\begin{theorem} \label{thm:worstcase}
  If $\ba_{=}( \cdot )$ is a Pareto-optimal control satisfying $\J_i ( \ba_{=} ( \cdot ) ) = \J_j ( \ba_{=} ( \cdot ) )$ for all $i,j \in \{ 1, \cdots r \}$, then $\ba_{=}(\cdot)$ is also worst-case optimal.
\end{theorem}
 \begin{proof}
   Suppose there exists $\ba ' ( \cdot )$ s.t.
   \[
   \max_{\hat{\x}_i \in \Op} \J_i ( \ba ' ( \cdot ) )  \; < \;  \max_{\hat{\x}_i \in \Op} \J_i ( \ba_{=} ( \cdot ) )
   \]
   then for all $j$ we have:
   \[
   \J_j (\ba' ( \cdot ))  \leq \max_{\hat{\x}_i \in \Op} \J_i (\ba' ( \cdot )) < \max_{\hat{\x}_i \in \Op} \J_i (\ba_{=} ( \cdot )) = \J_j (\ba_{=} ( \cdot )),
   \]
   which contradicts the Pareto-optimality of $\ba_{=}( \cdot).$
\end{proof}

\begin{figure}[tbhp]
\centering
\hspace{-0.4cm}
\subfloat[]{ \label{fig:obs1a}
\hspace{-0.4cm}
\begin{minipage}[b]{.285\textwidth}
  \vspace{-2cm}
\includegraphics[width=1\textwidth,keepaspectratio]{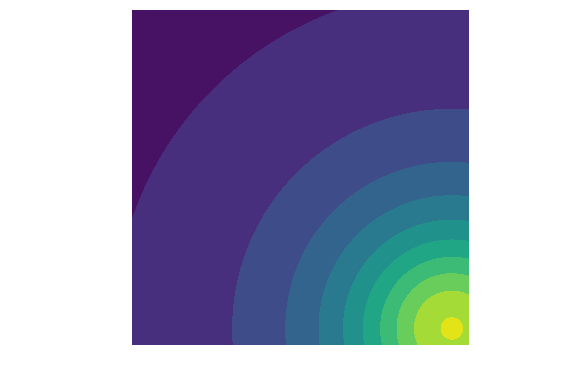} \\
\includegraphics[width=1\textwidth,keepaspectratio]{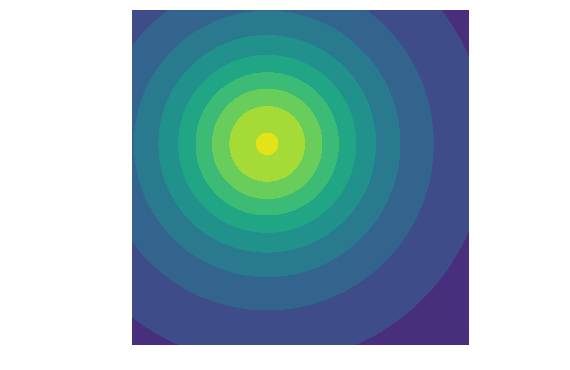}
\end{minipage}
}
\hspace{-0.55cm}
\subfloat[]{\label{fig:obs1b}
\includegraphics[width=0.375\linewidth,keepaspectratio]{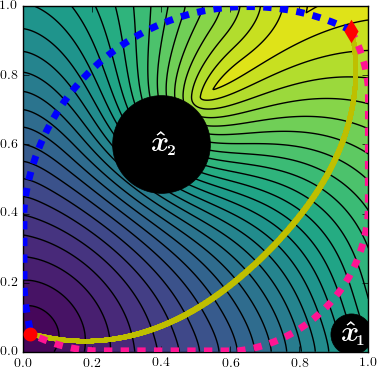}
}
\vspace{0.1cm}
\subfloat[]{\label{fig:obs1c}
  \begin{overpic}[width=0.355\linewidth,keepaspectratio]{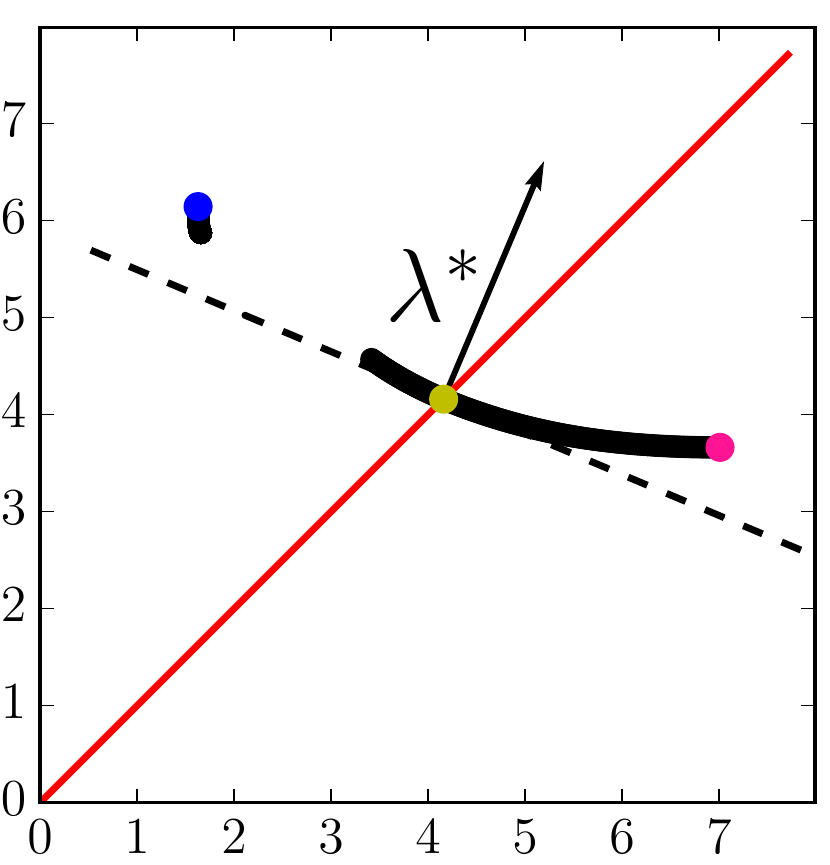}
  \put(-8,45) {\rotatebox{90}{$\J_2$}}
  \put(50,-2) {{$\J_1$}}
  \vspace{0.3cm}
\end{overpic}
\vspace{0.3cm}
}
\caption{
    (a) Two observer positions and the corresponding observability maps $\K_i$.
    (b) The $\Os^*$-optimal path corresponding to $\Os^* \approx (0.30, 0.70)$
    is shown in yellow over the level sets of $u^{\Os^*}$.
    The radii of black disks centered at $\xhat_i'$s are proportional to the corresponding components of $\bm{\Os^*}$. The two best response trajectories used when O chooses $\hat{\x}_1$ or $\hat{\x}_2$ are shown in blue and pink respectively. The trajectory in yellow is worst-case optimal for the evader as it is equally observable from both locations.
    (c) The convex part of Pareto Front (computed using the scalarization approach) intersects the ``central ray'' ($\J_1 = \J_2,$ shown in red).
    The worst-case optimal vector $\bm{\Os^*}$ is orthogonal to PF at the point of intersection (in yellow), whose coordinates correspond to the partial costs of the optimal path.
  The probability distribution $\Os^*$, together with the yellow trajectory form a Nash equilibrium of the strategic game between the evader and the observer described in~\cref{sec:ObserverEvader}.
  See~\cref{sec:examples} for additional information and parameters used.
}
\label{fig:observability1}
\end{figure}

\begin{figure}[tbhp]
\centering
\hspace{-0.4cm}
\subfloat[]{ \label{fig:NoObstaclesHolea}
\hspace{-0.4cm}
\begin{minipage}[b]{.285\textwidth}
  \vspace{-2cm}
\includegraphics[width=1\textwidth,keepaspectratio]{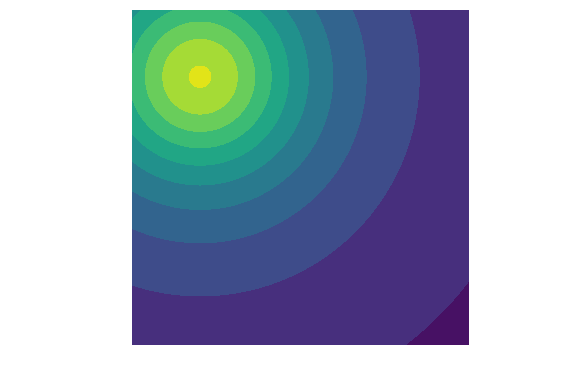} \\
\includegraphics[width=1\textwidth,keepaspectratio]{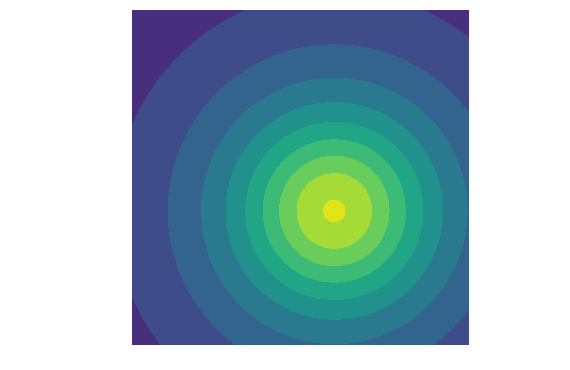}
\end{minipage}
%\vspace{5mm}
}
\hspace{-0.55cm}
\subfloat[]{\label{fig:NoObstaclesHole2b}
\includegraphics[width=0.375\linewidth,keepaspectratio]{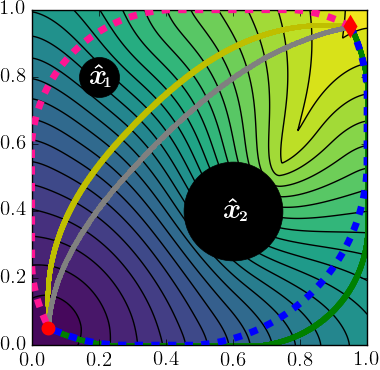}
}
\vspace{0.1cm}
\hspace{-0.15cm}
\subfloat[]{\label{fig:NoObstaclesHole2c}
\vspace{-0.1cm}
  \begin{overpic}[width=0.360\linewidth,keepaspectratio]{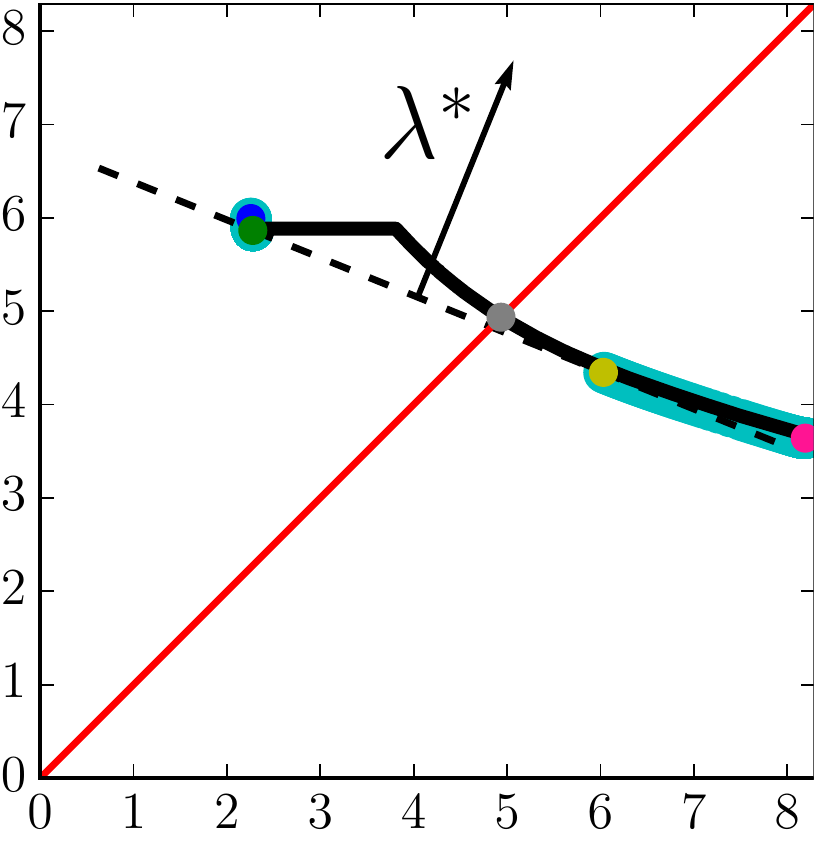}
  \put(-8,45) {\rotatebox{90}{$\J_2$}}
  \put(55,-5) {{$\J_1$}}
  \vspace{0.3cm}
\end{overpic}
\vspace{0.3cm}
}
\caption{
    (a) Two observer positions and the corresponding observability maps $\K_i$.
    (b) Two $\Os^*$-optimal trajectories corresponding to $\Os^* \approx (0.29, 0.71)$
    are shown in yellow and green over the level sets of $u^{\Os^*}$.
  The two best response trajectories used when O chooses $\hat{\x}_1$ or $\hat{\x}_2$ are shown in blue and pink respectively.
  The worst-case optimal trajectory is plotted in gray.
    (c) The convex part of the Pareto Front (in cyan) computed using the scalarization approach, and the whole Pareto Front (in black) computed using the method in~\cite{KumarVlad}.
  The convex part of the Pareto Front does not intersect the central ray (shown in red). The worst-case optimal strategy (in gray) lies on the non-convex part of the Pareto Front and thus cannot be computed using scalarization.
  The Nash equilibrium pair of strategies consists of using positions $\hat{\x}_1$ and $\hat{\x}_2$ with probabilities $\Os^*$ for O
   and using the yellow and green trajectories (both of which lie on the convex part of the PF) with probability $\left[ p( \text{yellow}),p(\text{green} ) \right]=\left[ 0.29, 0.71 \right] $ for E (see~\cref{sec:ObserverEvader}).  The latter mixed strategy is average-case optimal for E.
  See~\cref{sec:examples} for additional information and parameters used.}
\label{fig:NoObstaclesHole}
\end{figure}

The corresponding vector of costs $\bJ(\ba_{=}( \cdot ))$ may lie on the convex portion of PF, as in Figures~\ref{fig:pf_scal}(A) and \ref{fig:observability1}, in which case  $\ba_W = \ba_{=}$ can be found by scalarization \cite{MitchellSastry}.
But if $\bJ(\ba_{=}( \cdot ))$ lies on the non-convex portion of PF, as in Figures~\ref{fig:pf_scal}(B) and \ref{fig:NoObstaclesHole}, the computational cost of finding the evader's worst-case optimal strategy grows exponentially with $r$ as it involves solving a non-linear differential equation in $(r+d-1)$ dimensions \cite{KumarVlad}.
As it will be shown in sections  \ref{sec:ObserverEvader}-\ref{sec:MultipleEvaders},
the latter scenario is particularly common on domains with obstacles.

Luckily, another interpretation of evader's objectives proves much more computationally tractable.
Even though $\ba_{=} ( \cdot )$ yields the lowest {\em worst-case} observability that E can achieve if he must choose a single control function deterministically, E might be able to attain an even lower expected (or {\em average-case}) observability if he switches to ``mixed policies'', choosing a probability distribution over a set of Pareto optimal controls.
This is illustrated in~\cref{fig:pf_scal}(B): by choosing probabilistically a path corresponding to the point P and another corresponding to point R, E obtains a new point S on the central ray, whose expected observability is lower than for the worst-case optimal Q regardless of O's selected location.  This, of course, assumes that O's location is selected without knowing in advance which of the two paths will be used by E.
Indeed, for any single run from $\x_S$ to $\x_T$, the {\em worst-case} observability of this probabilistic policy is based on the worst cases for P and R, which (from the point of view of a completely risk-averse evader) would make the average-case optimal S inferior to the worst-case optimal Q.  This scenario is fully realized in \cref{fig:NoObstaclesHole}, where $J_1(\ba_{=}( \cdot )) = J_2(\ba_{=}( \cdot )) \approx 4.94$, the expected observability corresponding to the optimal ``probabilistic mix'' of yellow and green trajectories is $\approx 4.83$, but the worst case associated with this mixed policy is
$J_1(\text{yellow}) \approx 6.03$.

We note that O could also consider using a mixed strategy. In this case, $\K^{\Os}$ can be interpreted as the expected pointwise observability when using the probability density $\Os \in \Om$ over the positions $\Op$. Similarly $\J^{\Os}(\ba( \cdot ) )$ is the expected cumulative observability when using the control function $\ba ( \cdot )$.
\Cref{fig:pf_scal} shows that when we are interested in the average-case optimal behavior for both O and E, we
only need to consider a convex hull of PF (denoted $co(\text{PF})$), and the scalarization is thus adequate.  Note that in~\cref{fig:observability1,fig:NoObstaclesHole,fig:observability2}, the set $co(\text{PF})$ was approximated by imposing a fine grid on $\Om$ and re-solving \cref{eq:WeightedCostEikonal} for each sampled $\Os.$  Since we only care about the intersection of $co(\text{PF})$ with the central ray, this procedure is wasteful -- and prohibitively expensive for high $r$.
In the next section, we consider the case where both E and O optimize the expected/average-case performance by reformulating this as a semi-infinite strategic zero-sum game.  We show that such Surveillance-Evasion Games (SEGs) can be solved through scalarization combined with convex optimization, without approximating the (convex hull of the) entire Pareto Front.

\begin{remark}
Up till now, our geometric interpretation in~\cref{fig:observability1,fig:NoObstaclesHole,fig:observability2} assumed that either PF or at least the $co(\text{PF})$ must intersect the central ray.  If this is not the case, O will avoid using some of his positions.
E.g.,~\cref{fig:observability3} shows the pink and yellow trajectories corresponding to $\ba_1(\cdot)$ and $\ba_2(\cdot)$, which are optimal with respect to the observer positions $\hat{\x}_1$ and $\hat{\x}_2$.
Since $\J_1( \ba_2(\cdot) )\, \leq  \, \J_2( \ba_2(\cdot)),$ the E's worst-case for $\ba_2(\cdot)$ is actually the observer location $\hat{\x}_2$.
A generalization of this scenario for $r>2$ is covered in~\cref{thm:main}.
\end{remark}

\begin{figure}[tbhp]
\centering
\hspace{-0.4cm}
\subfloat[]{ \label{fig:obs3a}
\hspace{-0.4cm}
\begin{minipage}[b]{.27\textwidth}
  \vspace{-2cm}
\includegraphics[width=1\textwidth,keepaspectratio]{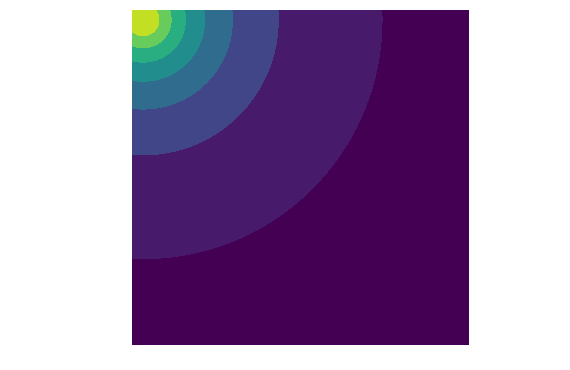} \\
\includegraphics[width=1\textwidth,keepaspectratio]{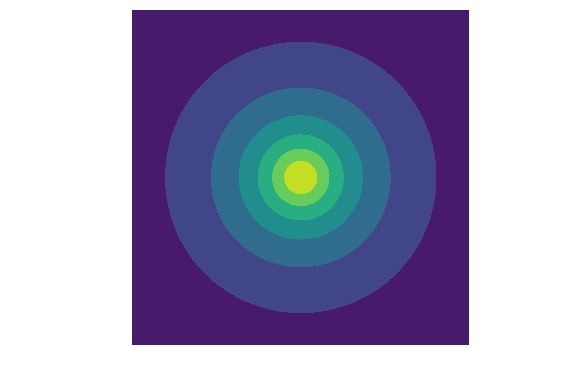}
\end{minipage}
}
\hspace{-0.55cm}
\subfloat[]{\label{fig:obs3b}
\includegraphics[width=0.360\linewidth,keepaspectratio]{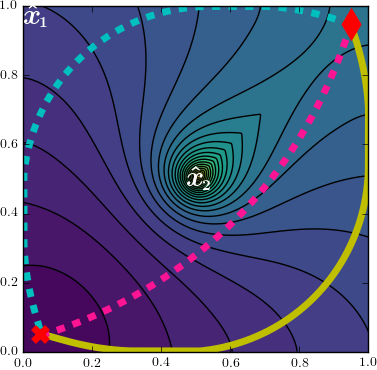}
}
\subfloat[]{\label{fig:obs3c}
  \begin{overpic}[width=0.355\linewidth,keepaspectratio]{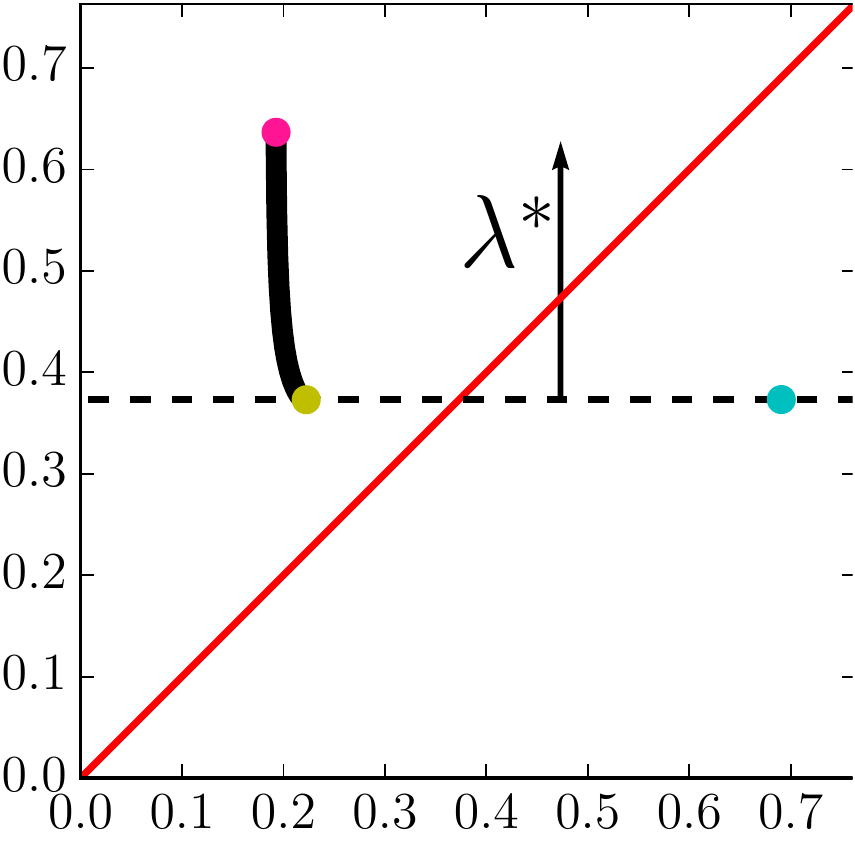}
  \put(-6,44) {\rotatebox{90}{$\J_2$}}
  \put(55,-5) {{$\J_1$}}
  \vspace{0.3cm}
\end{overpic}
\vspace{0.5cm}
}
\caption{ \footnotesize
    (a) Two observer positions and the corresponding observability maps $\K_i$ plotted in logarithmic scale.
    (b) The value function $u^{\Os^*}$ at $\Os^* = (0,1) $. The worst-case optimal strategy for O is the yellow trajectory, but both the yellow trajectory and the light blue trajectories are $\Os^*$-optimal. The pink trajectory is the best response when the observer uses position $\hat{\x}_1$.
    (c)  The Pareto Front does not intersect the central ray.
    The worst-case optimal trajectory is the one point on the Pareto Front that is closest to the central ray: the yellow point. The blue point is $\Os^*$-optimal but it is not Pareto optimal as it is dominated by the yellow point.
  The Nash equilibrium strategy consists of the position $\hat{\x}_2$ for O, and the yellow trajectory for E (see~\cref{sec:ObserverEvader}).
  See~\cref{sec:examples} for additional information and parameters used.
}
\label{fig:observability3}
\end{figure}

%% file: fig_pf_schematic.tex
%%%%%%%
\begin{figure}[hhhh]
\begin{center}
\def\tikzPfScale{0.6}
\begin{tabular}{c c c}
\input{pf_schematic_1.tex}
&

\input{pf_schematic_3.tex}
\\
(A)  & (B)
\end{tabular}
\end{center}
\caption{ \footnotesize
	\textsc{(a)} Convex smooth Pareto Front
	with a point $Q$ corresponding to the worst case optimal $\Os = (\lambda_1, \lambda_2) \in [0,1]^2$.
	The line perpendicular to $\Os$ is tangent to PF at $Q$.
	If any part of PF fell below it, the path corresponding to $Q$ would not be $\Os$-optimal. The dotted line is the central ray (where $\J_1 = \J_2)$).
	%\textsc{(b)} Convex non-smooth Pareto Front with a `kink' at the point $P$ makes the corresponding path $\Os$-optimal for a range of $\Os$s , with a different ``support hyperplane'' corresponding to each of them.
	\textsc{(b)} Non-convex smooth Pareto Front. Points $P$ and $R$ correspond to 2 different $\Os$-optimal paths.
	The portion of PF between $P$ and $R$ (including the worst-case optimal point $Q$) cannot be found by scalarization. 
	Point $S$, found as a convex combination of $P$ and $R$, is average-case optimal.
}
\label{fig:pf_scal}
\end{figure}
%%%%%%%

%% file: pf_schematic_1.tex
%===================================================================%
% !TEX root = ../multio_PRM.tex
%-------------------------------------------------------------------%

\begin{tikzpicture}[
point/.style={draw,shape=circle,inner sep = 0mm,minimum size = .2cm},
scale = \tikzPfScale,
%->,
>=stealth',
shorten >=1pt,
shorten <=1pt,
auto,
node distance = 1.5cm,
semithick,
 extended line/.style={shorten >=-#1,shorten <=-#1},
 extended line/.default=1cm,
 one end extended/.style={shorten >=-#1},
 one end extended/.default=1cm,
 tangent/.style={
		 decoration={
				 markings,% switch on markings
				 mark=
						 at position #1
						 with
						 {
								 \coordinate (tangent point-\pgfkeysvalueof{/pgf/decoration/mark info/sequence number}) at (0pt,0pt);
								 \coordinate (tangent unit vector-\pgfkeysvalueof{/pgf/decoration/mark info/sequence number}) at (1,0pt);
								 \coordinate (tangent orthogonal unit vector-\pgfkeysvalueof{/pgf/decoration/mark info/sequence number}) at (0pt,1);
						 }
		 },
		 postaction=decorate
 },
 use tangent/.style={
		 shift=(tangent point-#1),
		 x=(tangent unit vector-#1),
		 y=(tangent orthogonal unit vector-#1)
 },
 use tangent/.default=1
 ]

%%-- Settings ----------------------------------------------%%
% Size of objective space
\def\maxx{5}
\def\maxy{5}

%%------------------------------------------------------------------------------------%%
%%------------------------------------------------------------------------------------%%
%%------------------------------------------------------------------------------------%%
%\begin{scope}[shift={(0,0)}, scale=1.0]

%%-- Pareto Front curve ------------------------------------%%
% Coordinates along curve:
\coordinate (X1) at (0.25*\maxx, 1.00*\maxy);
%\coordinate (X2) at (0.20*\maxx, 0.60*\maxy);
%\coordinate (X3) at (0.50*\maxx, 0.55*\maxy);
%\coordinate (X4) at (0.65*\maxx, 0.25*\maxy);
\coordinate (Xn) at (1.00*\maxx, 0.05*\maxy);

% Curve itself:
\draw[thick, tangent=0.455]
	(X1) to[out=-75,in=175]
	(Xn);

% Points on curve
\node[point,fill=black,scale=0.7] at (X1) {};
\node[point,fill=black,scale=0.7] at (Xn) {};

%%-- Axes -------------------------------------------------%%
% Coordinates
\coordinate (H1) at (-1,0);
\coordinate (H2) at (\maxx+1,0);
\coordinate (V1) at (0,-1);
\coordinate (V2) at (0,\maxy+1);
\coordinate (O) at (0,0);  % origin
\coordinate (UR) at (1*\maxx,1*\maxy);  % origin

% Draw them
\draw[->] (H1) -- (H2);
\draw[->] (V1) -- (V2);

% Axes labels
\node[right] (T) at (H2) {\footnotesize $\J_1$};
\node[below left] (D) at (V2) {\footnotesize $\J_2$};

%%-- Lambda vector -------------------------------------------------%%
\coordinate (PL) at (0.39*\maxx + 1.0, 0.3*\maxy + 1.0);
\draw[->, thick, use tangent] (0,0) -- (0,1);

%%-- Dashed lines -------------------------------------------------%%
\draw[ dotted, black] (O) -- (UR);
\draw[dashed, color=gray, use tangent] (-3,0) -- (2,0);

%%-- Points on curve: P --------------------------------------%%
\coordinate[use tangent] (P) at (0,0);
\node[point,scale=0.7] (Ppt) at (P) {};
\node[below left]       (Plb) at (P) {$Q$};
\node[right] (PLlb) at (PL) {\scriptsize $\bm{\lambda}$};
\node[above]  (URar) at (UR) {$\J_1 = \J_2$};

% P to R line
%  \tkzDefPoint(0.38*\maxx, 0.31*\maxy){Ppgf}
%  \tkzDefPoint(0.40*\maxx, 0.29*\maxy){Rpgf}
%  \tkzDefLine[parallel=through Ppgf](Ppgf,Rpgf)
%  \tkzDrawLine[add = 25 and 25, color=gray, dashed](Ppgf,tkzPointResult)

%\end{scope}
\end{tikzpicture}

%% file: pf_schematic_3.tex
%===================================================================%
% !TEX root = ../multio_PRM.tex
%-------------------------------------------------------------------%

\begin{tikzpicture}[
point/.style={draw,shape=circle,inner sep = 0mm,minimum size = .2cm},
scale = \tikzPfScale,
%->,
>=stealth',
shorten >=1pt,
shorten <=1pt,
auto,
node distance = 1.5cm,
semithick,
 extended line/.style={shorten >=-#1,shorten <=-#1},
 extended line/.default=1cm,
 one end extended/.style={shorten >=-#1},
 one end extended/.default=1cm,]

%%-- Settings ----------------------------------------------%%
% Size of objective space
\def\maxx{5}
\def\maxy{5}

%%-- Points on curve: P, Q, R --------------------------------------%%
%\coordinate (P) at (X2);
\coordinate (P) at (0.21*\maxx, 0.64*\maxy);
%\coordinate (Q) at (X3);
\coordinate (R) at (0.83*\maxx, 0.105*\maxy);

% point Q: expected
\coordinate (Qhat) at (0.5*0.125*\maxx+0.5*0.7*\maxx , 0.5*0.635*\maxy+0.5*0.16*\maxy);
\coordinate (Q) at (1.32*0.5*0.125*\maxx+1.32*0.5*0.7*\maxx , 1.3*0.5*0.635*\maxy+1.3*0.5*0.16*\maxy);

\node[point,fill=white,scale=0.7] (Ppt) at (P) {};
\node[below left]       (Plb) at (P) {$P$};
%\node[point,fill=white,scale=0.7] (Qhatpt) at (Qhat) {};
%\node[below]       (Qhatlb) at (Qhat) {$\tilde{Q}$};
%\node[point,fill=white,scale=0.7] (Q) at (Q) {};
%\node[above]       (Qlb) at (Q) {${Q}$};

%\node[point,fill=black] (Qpt) at (Q) {};
%\node[above]            (Qlb) at (Q) {$Q$};
\node[point,fill=white,scale=0.7] (Rpt) at (R) {};
\node[above right]      (Rlb) at (R) {$R$};

%%-- Pareto Front curve ------------------------------------%%
% Coordinates along curve:
\coordinate (X1) at (0.10*\maxx, 1.00*\maxy);
\coordinate (X2) at (0.30*\maxx, 0.60*\maxy);
\coordinate (X3) at (0.60*\maxx, 0.55*\maxy);
\coordinate (X4) at (0.70*\maxx, 0.35*\maxy);
\coordinate (Xn) at (1.00*\maxx, 0.05*\maxy);

% Curve itself:
\draw[thick,name path=curve 1]
	(X1) to[out=-80,in=170]
	(X2) to[out=-10,in=135]
	(X3) to[out=-45,in=100]
	(X4) to[out=-80,in=180]
	(Xn);

% Points on curve
\node[point,fill=black,scale=0.7] at (X1) {};
\node[point,fill=black,scale=0.7] at (Xn) {};

%%-- Axes -------------------------------------------------%%
% Coordinates
\coordinate (H1) at (-1,0);
\coordinate (H2) at (\maxx+1,0);
\coordinate (V1) at (0,-1);
\coordinate (V2) at (0,\maxy+1);
\coordinate (O) at (0,0);  % origin
\coordinate (UR) at (1*\maxx,1*\maxy);  % origin

% Draw them
\draw[->] (H1) -- (H2);
\draw[->] (V1) -- (V2);

% Axes labels
\node[right] (T) at (H2) {\footnotesize $\J_1$};
\node[below left] (D) at (V2) {\footnotesize $\J_2$};

%%-- Dashed lines -------------------------------------------------%%
\draw[extended line, dashed, gray,name path=tangent] (P) -- (R);
\draw[dotted, black,name path=ray] (O) -- (UR);
\node[above]  (URar) at (UR) {$\J_1 = \J_2$};

%-- Draw intersection
\fill[point,fill=none,scale=0.7, name intersections={of=tangent and ray,total=\t}]
  (intersection-1) circle (5pt) node {};
  \node[below]      (Slb) at (intersection-1) {$S$};

  %-- Draw intersection
\fill[point,fill=none,scale=0.7, name intersections={of=curve 1 and ray,total=\t}]
  (intersection-1) circle (5pt) node {};
  \node[above]      (Qlb) at (intersection-1) {$Q$};

%    \foreach \s in {1,...,\t}{(intersection-\s) circle (5pt) node {}};
%    \node[above right] (Qar) at ((intersection-\s)) {$Q$};

%\end{scope}
\end{tikzpicture}

%% file: ObserverEvader.tex
% !TEX root = draft.tex

\section{Surveillance-Evasion Games (SEGs)} \label{sec:ObserverEvader}
In this section, we reformulate the problem of evasive path planning under surveillance uncertainty as a strategic game.
This can model either the actual adversarial interactions between two players or the risk-averse logic of the evader even if the surveillance patterns are not likely to change in response to that evader's strategy.  (The latter case is typically interpreted as a ``game against nature''.)

We assume that the evader is attempting to minimize (while the observer is attempting to maximize) the total expected observability integrated over E's trajectories and dependent on O's positions.
We further assume that O is aware of E's initial location $\x_S$ and its target location $\x_T$ but not of the trajectories chosen by E. Similarly, E is aware of the predefined locations of O, but not of the realized positions chosen by O.
This game may be stated deterministically or stochastically. In the deterministic case, each player chooses a single pure strategy.
That is, the observer chooses a single location $\hat{\x}_i \in \Op$ and the evader chooses a single control function $\ba ( \cdot ) \in \Ep$.
In the probabilistic setting, each player chooses a mixed strategy, i.e., a probability distribution over the pure strategies. In other words, O chooses a probability distribution $\Os \in \Om$ over positions and E chooses a probability distribution $\Es \in \Em$ over control functions.
The mixed strategy $\Os$ of the observer can be interpreted in several different ways:
\begin{enumerate}
\item O chooses a single position $\hat{\x}_i$ according to the probability distribution $\Os$ before E starts moving, and remains at that position until the end of the round (that is, until E reaches the target).
\item O can randomly teleport between its positions at any time, and each $\lambda_i$ reflects the proportion of time spent at the corresponding position $\hat{\x}_i$.
\item O has a budget of ``observation resources", and $\Os$ reflects the fraction of these resources spent at each location.
In this case,  $K_i$ reflects the pointwise observability corresponding to 100\% of resources allocated to the position $\hat{\x}_i$.
\end{enumerate}
Since we assume that neither player has access to the realization of the opponent's strategy in real time, these three interpretation are equivalent (and lead to the same optimal strategies) in our context.
The payoff function of the game is the cumulative expected observability, and can be expressed as $P( \Os,  \Es ) = \mathbb{E}_{\Es} \left[ \J^{\Os} ( \ba ( \cdot ) )\right]$ where $\mathbb{E}_{\Es} \left[ \cdot \right]$ denotes the expectation over the mixed strategy $\Es$.

This SEG is a two-player \textit{zero-sum game}~\cite{osborne1994course}, as each player's gains or losses are exactly balanced by the losses or gains of the opponent. Furthermore, it is \textit{semi-infinite} as the set of pure strategies for O is a finite number $r$, whereas the set of pure strategies for E is uncountably infinite. A central notion of solution for strategic games is a \text{Nash equilibrium}~\cite{osborne1994course}, a pair of strategies for which neither player can improve his payoff by unilaterally changing his strategy. That is, a pair of strategies $ ( \Os^* , \Es^* )$ is a Nash equilibrium if both of the following conditions hold:
\begin{equation} \label{eq:Nash}
  \begin{aligned}
P(\Os^*, \Es^*  ) \leq  P(\Os^* , \Es ) & \text{ for all }  \Es \in \Em \ , \\
P(\Os^*, \Es^*  ) \geq  P(\Os , \Es^*) & \text{ for all }  \Os \in \Om \ .
\end{aligned}
\end{equation}

A pure strategy Nash equilibrium does not always exist, therefore we focus on finding a mixed strategy Nash equilibrium. In our setting, the minimax theorem for semi-infinite games~\cite{raghavan1994zero} assures that a mixed strategy Nash equilibrium $( \Os^*, \Es^*)$ exists, that all Nash equilibria have the same payoff, and that they are attained at the minimax (which is also equal to the maximin):
\begin{equation} \label{eq:mixed}
 P(\Os^*, \Es^*) = \min_{  \Es \in \Em }  \max_{ \Os \in \Om } \mathbb{E}_{\Es} \left[ \J^{\Os}( \ba ( \cdot ) )  \right] = \max_{ \Os \in \Om } \min_{  \Es \in \Em }  \mathbb{E}_{\Es} \left[ \J^{\Os}( \ba ( \cdot ) )  \right] \ .
\end{equation}
Although $\Es$ is a probability distribution over the uncountable set $\Em$, there always exists an optimal mixed strategy $\Es^*$ which is a mixture of at most $r$ pure strategies, where $r$ is the maximum number of positions for the observer~\cite{raghavan1994zero}. In fact, it is easy to show that there will always exist a Nash equilibrium
$( \Os^*, \Es^*)$ with the number of pure strategies used in $\Es^*$ not exceeding the number of non-zero components in $\Os^*.$

In the case of finite two-player zero-sum games, computing the Nash equilibrium is easily achieved by linear programming. For our SEGs, the challenge in computing a Nash equilibrium arises from enumerating the control functions $\ba ( \cdot ) \in \Ep$ which are part of E's mixed strategy. Indeed, we do not possess a useful parametrization
of the set of control functions $\Ep$, and our only computational kernel to generate a single $\Os$-optimal control function $\ba^{\Os} ( \cdot )$ is to solve the weighted-cost Eikonal equation in~\cref{eq:WeightedCostEikonal}.
For that reason, our solution strategy to compute the Nash Equilibrium involves two steps:
\begin{enumerate}
  \item Find an approximate optimal strategy of the observer $\Os^*$ using convex optimization (see~\cref{sec:Ostrategy}).
  \item Find an approximate optimal strategy of the evader $\Es^*$ by generating near-optimal control functions (see~\cref{sec:Estrategy}).
\end{enumerate}

\subsection{Optimal strategy of the Observer}\label{sec:Ostrategy}

In order to compute an optimal strategy $\Os^*$ of the observer, we consider the following problem:
\begin{equation} \label{eq:pureE}
\max_{\Os \in \Om } \min_{ \ba( \cdot ) \in \Ep } \J^{\Os}\left( \x_S,\bap  \right) = \max_{\Os \in \Om} u^{\Os} ( \x_S)\ .
\end{equation}
For any fixed strategy $\Os$ for O, the inner minimization represents the optimal response of player E to that fixed strategy. Therefore, the maximin problem answers the question: what is the optimal strategy for O given that E gets to observe that strategy and respond? We call this problem the \textit{E-response} problem.
Note that although E could use a mixed strategy, there always exists a pure strategy which is optimal. That is:
\begin{equation} \label{eq:PureEqualsMixed}
  \min_{  \theta \in \Em }  \mathbb{E}_{\theta} \left[ \J^{\Os}( \ba ( \cdot ) )  \right] =  \min_{ \ba( \cdot )\in  \Ep } \J^{\Os}\left( \ba( \cdot ) \right) \ .
\end{equation}
This implies that any optimal $\Os$ for \cref{eq:pureE} is also an optimal $\Os $ for~\cref{eq:mixed}. Consequently, the optimal $\Os$ for~\cref{eq:pureE} is one half of a Nash equilibrium pair. However, the optimal pair $( \Os, \ba ( \cdot )) $ of~\cref{eq:pureE} is not a Nash equilibrium, except in a specific situation described in the following theorem.
\begin{theorem} \label{thm:purenash}
 Suppose there exists $ \Os_{=} \in \Om$ with associated $\Os_{=}$-optimal control function $\ba^{\Os_{=}} ( \cdot ) $ which satisfies $\J_i (\ba^{\Os_=} ( \cdot ) ) = \J_j ( \ba^{\Os_=} ( \cdot ) )$ for all $i,j \in \{ 1, \dots , r \}$, then $( \Os_=, \ba^{\Os_=}( \cdot ) ) $ is a Nash equilibrium.
\end{theorem}
 \begin{proof}
The fact that E cannot improve his payoff follows from the definition of $\ba^{\Os_=} \in \argmin_{\ba ( \cdot )} \J^{\Os_=}( \ba ( \cdot ) ) $. O may not improve his payoff either as for all $\Os$,
\[
\J^{\Os}( \ba^{\Os_=} ( \cdot ) ) = \sum \Os_i \J_i ( \ba^{\Os_=} ( \cdot ) ) = \sum \Os_{=,i} \J_i ( \ba^{\Os_=} ( \cdot ) ) = \J^{\Os_{=}}( \ba^{\Os_{=}} ( \cdot ) ) \ .
\]
\end{proof}
This situation corresponds to the case when the convex part of the Pareto Front intersects the central ray, such as in the example in~\cref{fig:observability1}.
\Cref{thm:worstcase} implies that in this case, the worst-case optimal strategy for E coincides with E's half of the Nash equilibrium.
Note that in general such a $\Os_{=}$ does not have to exist; e.g., in~\cref{fig:NoObstaclesHole} and~\cref{fig:observability2} the convex part of the Pareto Front does not intersect the central ray. In such situations, the worst-case optimal strategy for E and the Nash Equilibrium are different.
Moreover, the latter involves a mixed strategy for E covered in~\cref{sec:Estrategy}.

We now direct our attention to solving the E-response problem numerically.~\Cref{eq:pureE} may be stated as the following optimization problem:

\begin{align} \label{eq:maxG}
&\max_{\Os} G(\Os) \\
&\mbox{ s.t. } \lambda_i \geq 0, \quad \sum_{i=1}^r \lambda_i = 1 \ . \nonumber
\end{align}

The objective function $G(\Os) = \min_{\ba( \cdot ) \in \Ep } \sum_{i=1}^r \Os_i \J_i(  \ba ( \cdot ) )$ is concave as it is the pointwise minimum of linear functions.
Furthermore, the vector of individual cumulative costs $ \bJ( \ba^{\Os}(\cdot)  ) $, where $\ba^{\Os} ( \cdot)  \in \argmin_{\ba( \cdot ) \in \Ep } \J^{\Os}( \ba ( \cdot ) )$, is a supergradient of $G$ (denoted as  $ \bJ(  \ba^{\Os} ( \cdot ) )\in \partial G ( \Os )$). A supergradient provides an ascent direction of a concave function, i.e., $\w \in \partial G ( \Os ) $
 if for all  $\hat{\Os} \in \Om$,
\[
G(\hat{\Os}) - G(\Os) \leq \w^T(\hat{\Os} - \Os ) \ .
\]
The fact that $ \bJ(  \ba^{\Os}( \cdot ) )\in \partial G ( \Os )$ is seen from the following computation: for any $\hat{\Os}$,
\begin{align*}
G(\hat{\Os} ) - G(  \Os ) =  \left(  \min_{ \ba \in \Ep}\sum_{i=1}^r \hat{\lambda}_i \J_i(  \ba( \cdot) ) \right) - \sum_{i=1}^r  \lambda_i\J_i( \ba^{\Os}( \cdot) ) \\
 \leq   \sum_{i=1}^r \hat{\lambda}_i \J_i( \ba^{\Os}( \cdot) ) - \sum_{i=1}^r \lambda_i\J_i(  \ba^{\Os}( \cdot) ) \\
= \bJ(  \ba^{\Os}(\cdot) )^T (\hat{\Os} - \Os)   \ .
\end{align*}
Evaluating the vector $\bJ ( \ba^{\Os}( \cdot))$ can be challenging computationally; we show how this can be done in~\cref{sec:computingcosts}.
Once this ascent direction is known, one still needs to ensure that $\Os$ remains a feasible probability distribution over $\Op$, and we use the orthogonal projection operator $\Pi: \R^r \rightarrow \Om$.
The operator $\Pi$ can be computed in $\mathcal{O}(r \log r) $
 operations~\cite{brucker1984n, wang2013projection} as summarized in~\cref{alg:projection}.
The resulting projected supergradient method~\cite[Chap. 8]{beck2017first} is shown in~\cref{alg:supergradient}.
The iterates of~\cref{alg:supergradient} for the example from~\cref{fig:observability2} are illustrated in~\cref{fig:AlgIterates}.

\begin{algorithm}[h]
\begin{algorithmic}[1]
\caption{ Orthogonal projection onto the probability simplex } \label{alg:projection}
\STATE{ {\bf Input}  $\Os \in \mathbb{R}^r$ }
\STATE{Sort $\Os$ into $\bm{u}$: $u_1 \geq u_2\geq  \dots  \geq u_r $ }
\STATE{ Find $\rho = \max \{ 1 \leq j \leq r :  u_j + \frac{1}{j} \left( 1 -  \sum_{i=1}^j u_i \right) > 0 $\} }
\STATE{ $\tau \gets \frac{1}{\rho}\left( 1 - \sum_{i=1}^{\rho} u_i \right) $}
\RETURN $x$ s.t. $x_i = \max \{ \lambda_i + \tau, 0 \} $, $i = 1, \dots r $.
\end{algorithmic}
\end{algorithm}

\begin{algorithm}[h]
\begin{algorithmic}[1]
\caption{Projected supergradient method for finding the maximum of $G$ over the set $\Om$ }\label{alg:supergradient}
\STATE{{\bf Input} Initial guess $\Os_0$, stepsizes $\alpha_k $, number of iterations $n$}
\FOR{ $k = 0 : (n-1)$}
\STATE{Compute $G(\Os_k) = u^{\Os_k}(\x_S)$ and
find some $\bm{g} \in \partial G ( \Os_k) $ }
\STATE{ $\Os_{k+1} \gets \Pi \left( \Os_{k} + \alpha_k \bm{g} \right) $ }
\ENDFOR \\
\RETURN $\argmax\limits_{\Os \in \left\{\Os_0,\ldots, \Os_n \right\} } G(\Os)$
\end{algorithmic}
\end{algorithm}

\begin{figure}[tbhp]
\centering
\hspace{-0.4cm}
\subfloat[]{ \label{fig:obs2a}
\hspace{-0.4cm}
\begin{minipage}[b]{.285\textwidth}
  \vspace{-2cm}
\includegraphics[width=1\textwidth,keepaspectratio]{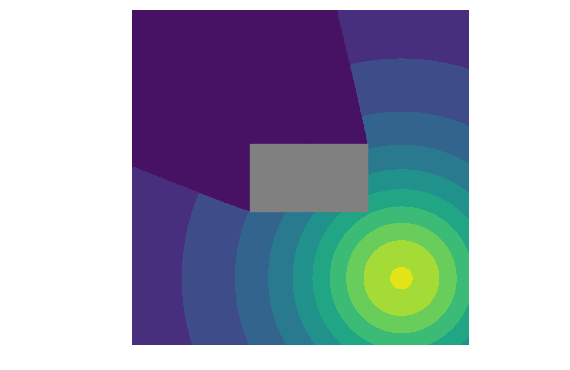} \\
\includegraphics[width=1\textwidth,keepaspectratio]{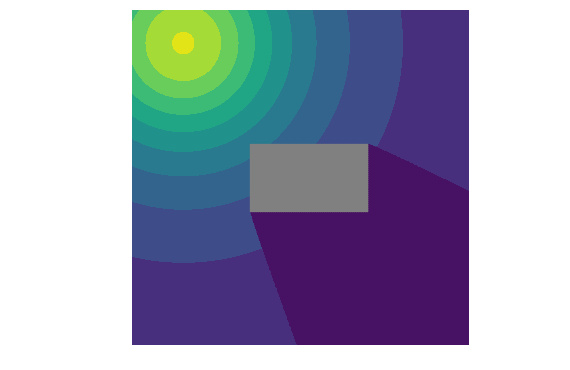}
\end{minipage}
%\vspace{5mm}
}
\hspace{-0.55cm}
\subfloat[]{\label{fig:obs2b}
\includegraphics[width=0.38\linewidth,keepaspectratio]{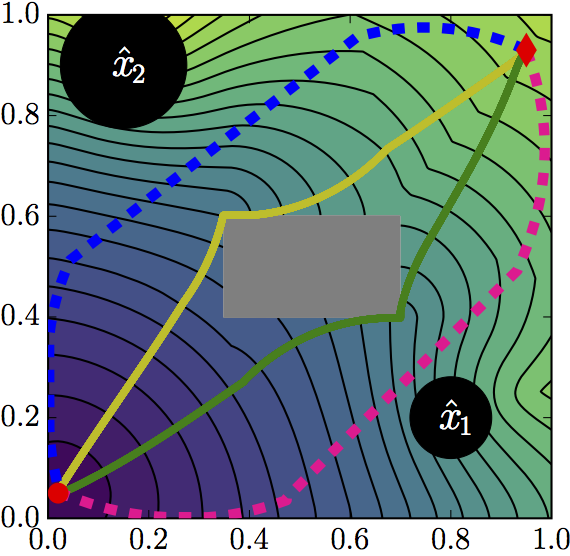}
}
\vspace{0.1cm}
\subfloat[]{\label{fig:obs2c}
  \begin{overpic}[width=0.35\linewidth,keepaspectratio]{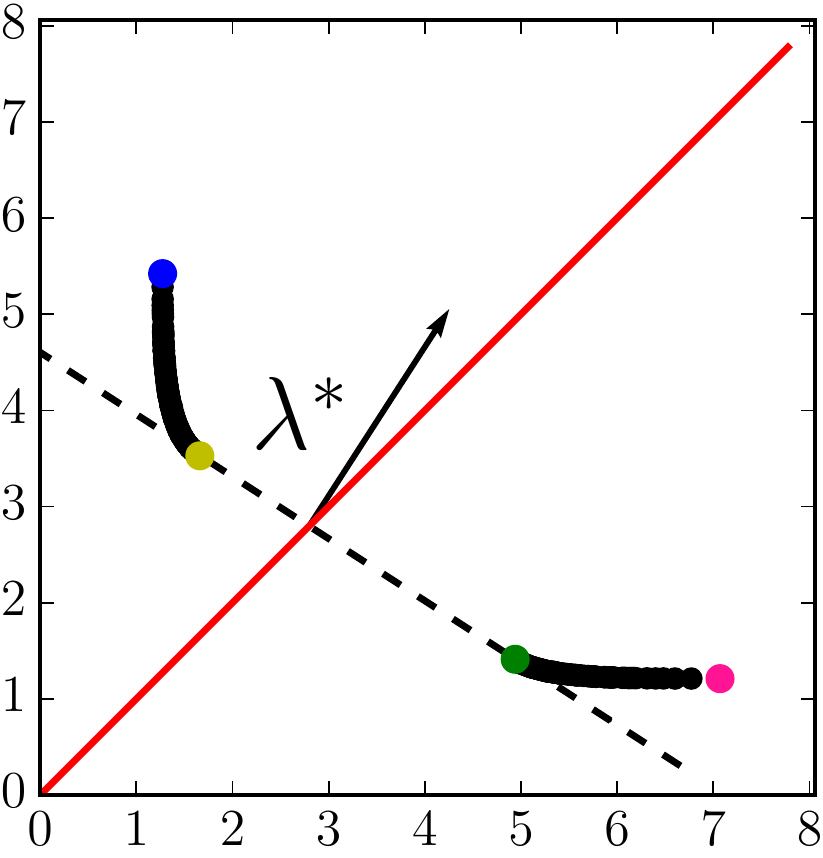}
  \put(-8,45) {\rotatebox{90}{$\J_2$}}
  \put(50,-2) {{$\J_1$}}
  \vspace{0.3cm}
\end{overpic}
\vspace{0.3cm}
}
\caption{
    (a) Two observer positions and the corresponding observability maps $\K_i$ on a domain with a single obstacle (shown in gray).
    (b) Two $\Os^*$-optimal trajectories corresponding to $\Os^* \approx (0.39, 0.61)$
    are shown in yellow and green over the level sets of $u^{\Os^*}$.
  The two best response trajectories used when O chooses $\hat{\x}_1$ or $\hat{\x}_2$ are shown in blue and pink respectively. The trajectories in yellow and green are not worst-case optimal for the evader but are used in E's mixed Nash equilibrium strategy.
    (c) The convex part of the Pareto Front does not intersect the central ray (shown in red).
    This is the same situation already observed in~\cref{fig:NoObstaclesHole}, but it is even more common on domains with obstacles.
  The Nash equilibrium pair of strategies consists of using positions $\hat{\x}_1$ and $\hat{\x}_2$ with probabilities $\Os^*$ for O,
   and using the yellow and green trajectories with probability $\left[ p( \text{yellow}),p(\text{green} ) \right]=\left[ 0.65,  0.35 \right] $ for E. %(see~\cref{sec:ObserverEvader}).
  See~\cref{sec:examples} for additional information and parameters used.}
\label{fig:observability2}
\end{figure}

\begin{figure}[tbhp]
\centering
\subfloat[]{\label{fig:subgradienta}
\begin{overpic}[width=0.46\textwidth]{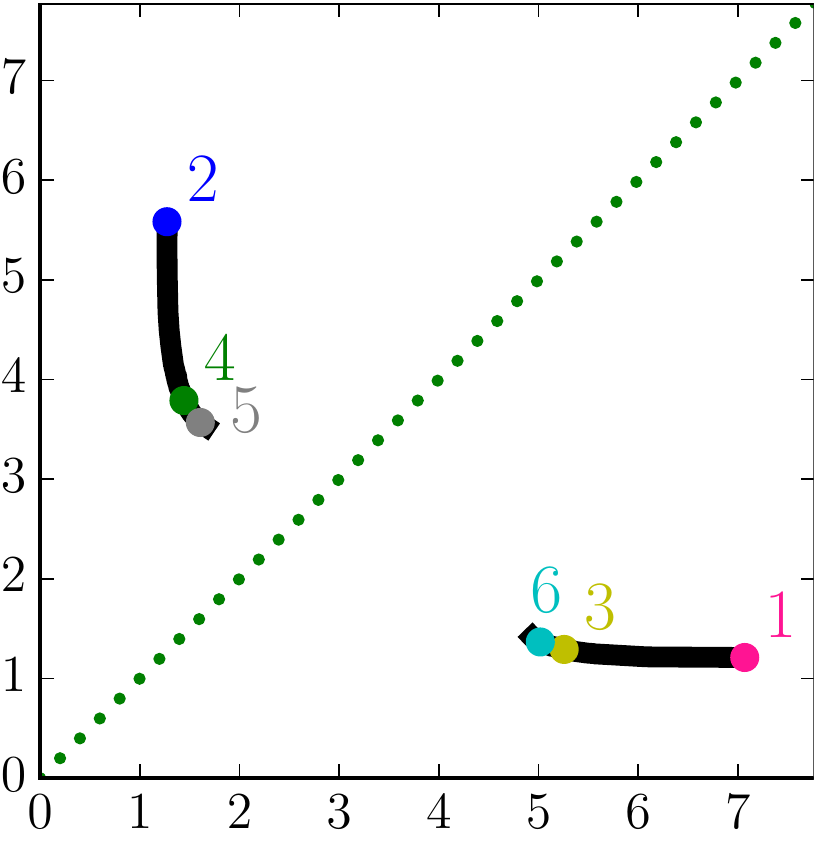}
  \put(-8,45) {\rotatebox{90}{$\J_2$}}
  \put(53,-3) {{$\J_1$}}
\end{overpic} }
\subfloat[]{\label{fig:sugradientb}
\includegraphics[width=0.5\textwidth]{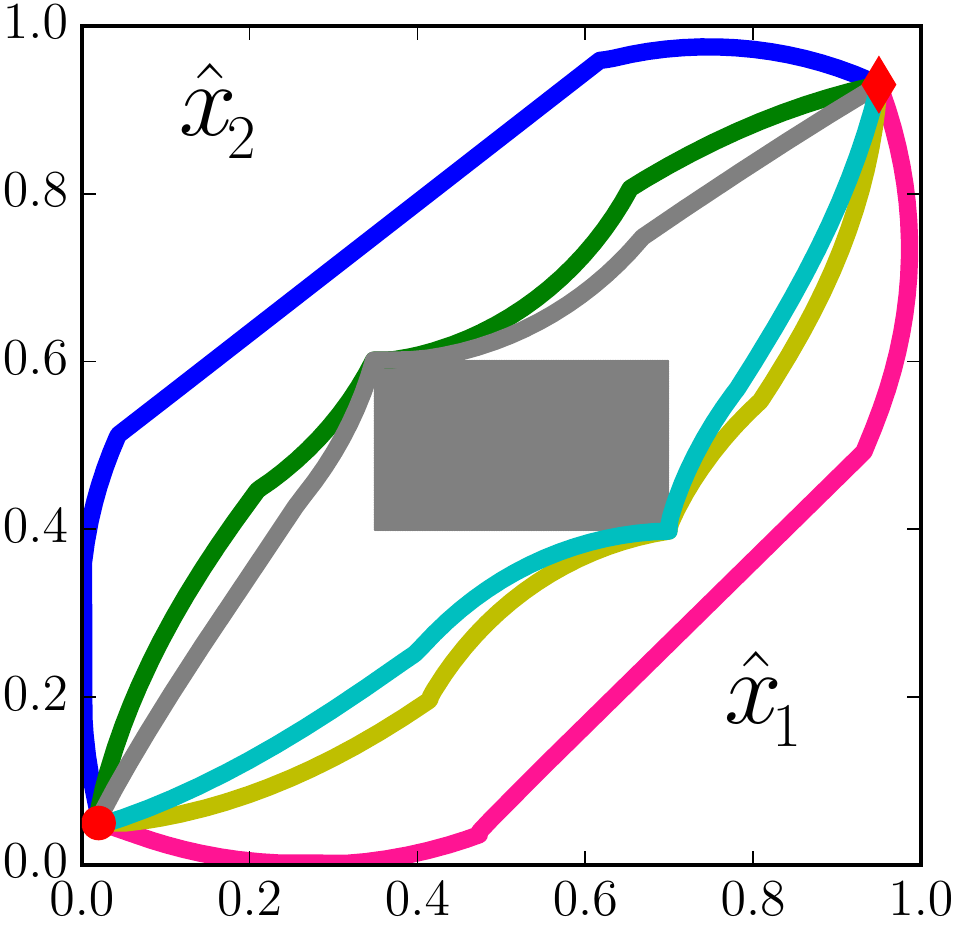}
 }
\caption{ \footnotesize
(a) Convex part of PF, and the individual costs of the first 6 iterates $\Os_k$ of~\cref{alg:supergradient} (with stepsizes $\alpha_k = 3 \big/  k$) for the problem shown in~\cref{fig:observability2}. (b) The $\Os_k$-optimal trajectories of the first 6 iterates.
We note that only a few iterates are needed to obtain trajectories which are close to the central ray. Thus, it does not require computing the whole PF which saves computational time.
}
\label{fig:AlgIterates}
\end{figure}

\subsection{ Optimal strategy of the Evader} \label{sec:Estrategy}
Computing the evader's half of the Nash equilibrium is more challenging due to the fact that the set of E's pure strategies, i.e., the set of control functions $\ba( \cdot )$ leading from the source $\x_S$ to the target $\x_T$, is uncountably infinite. We propose a heuristic strategy to approximate $\Es^*$ which relies on two properties of the Nash equilibrium in semi-infinite games:
\begin{enumerate}
\item There exists a Nash mixed strategy for E which uses only $r$ pure strategies\footnote{This result assumes that the set $S=\{ (s_1, s_2, \dots s_r) \mid s_i = P(\hat{\x}_i, \ba( \cdot )); i = 1,2,\dots ,r; \ba( \cdot )\in \Ep \} \subset \mathbb{R}^r $ is bounded and $co(S)$ is closed. In our case, $S$ is not bounded for the full set of control functions in $\Ep$ but becomes bounded if we restrict our attention to Pareto optimal control functions. }~\cite{raghavan1994zero}.
\item All pure strategies employed with positive probability in the Nash equilibrium have the same expected payoff, with the expectation taken over the other half of the Nash. In particular, all control functions used with positive probability in the Nash equilibrium are $\Os^*$-optimal.
\end{enumerate}

The following characterization of the Nash equilibrium helps us generate a good candidate set of $\Os^*$-optimal trajectories.

\begin{theorem}\label{thm:main}
  Let $( \Os^*, \Es^*)\in \Om \times \Em $ and $\mathcal{I} = \{ i \mid \lambda_i^* \neq 0 \}$. $( \Os^*, \Es^*)$ is a Nash equilibrium if and only if the following three conditions hold:
\begin{enumerate}
  \item $\Os^*$ is a constrained maximizer of $G(\Os)$ in ~\cref{eq:maxG}, \label{thm:main1}
  \item if $i \in \mathcal{I}$ then $\mathbb{E}_{\Es^*}\left[ \J_i( \ba ( \cdot ) )\right] = G( \Os^*) $, and \label{thm:main2}
  \item if $i \not\in \mathcal{I}$, then $\mathbb{E}_{\Es^*}\left[ \J_i( \ba ( \cdot ) )\right] \leq G( \Os^*)$.\label{thm:main3}
\end{enumerate}
\end{theorem}
\begin{proof}
( $\Rightarrow $)

Suppose $(\Os^*, \Es^*)$ is a Nash equilibrium. \Cref{thm:main1} follows from the minimax theorem for semi-infinite game and ~\cref{eq:PureEqualsMixed}. Assume \cref{thm:main2} does not hold,
then there must exist $i,j \in \mathcal{I}$ s.t. $\mathbb{E}_{\Es^*}\left[ \J_i( \ba ( \cdot ) )\right] >\mathbb{E}_{\Es^*}\left[ \J_j(\ba ( \cdot ) )\right]$. Consider the strategy $\hat{\Os} \in \Delta_r $:
\[
\hat{\lambda}_k  = \begin{cases}
\lambda^*_i + \lambda^*_j & \text{if } k = i \\
0 & \text{if } k = j \\
\lambda^*_k & \text{otherwise } \\
\end{cases} \ .
\]
Then we have that:
\[
\tiny
P( \Os^*, \Es^*)
= \sum_{i=1}^{r } \lambda^*_i \mathbb{E}_{\Es^*} \left[ \J_i( \ba ( \cdot ) )\right] < \sum_{i=1}^{r } \hat{\lambda}_i \mathbb{E}_{\Es^*}\left[ \J_i( \ba ( \cdot ) )\right] = P(\hat{\Os}, \Es^*) \ .
\]
This contradicts that $(\Os^*, \Es^*)$ is a Nash equilibrium, thus~\cref{thm:main2} must hold.
A similar argument can be used to demonstrate~\cref{thm:main3}: assume there exists $i \not\in \mathcal{I}$ with $\mathbb{E}_{\Es^*}\left[ \J_i( \ba ( \cdot ) )\right] > G( \Os^*) $. Let $j \in \mathcal{I}$ and consider the strategy $\hat{\Os}$:
\[
\hat{\lambda  }_k  = \begin{cases}
\lambda^*_j & \text{if } k = i \\
0 & \text{if } k = j \\
\lambda^*_k & \text{otherwise } \\
\end{cases}
\]
Once again, this implies that $ P( \Os^*, \Es^*)  < P(\hat{\Os}, \Es^*)$ which contradicts that $( \Os^*, \Es^*)$ is a Nash equilibrium. \

($\Leftarrow $)\
Assume~\cref{thm:main1,thm:main2,thm:main3} hold and suppose
there exists $\Es$ s.t. $P(\Os^*, \Es) < P(\Os^*, \Es^*)$, then there must exist $\ba( \cdot )$, used with non-zero probability in $\Es$ such that:
\[
\J^{\Os^*} ( \ba ( \cdot )) < P(\Os^*, \Es^*) = G(\Os^*) \ .
\]
This contradicts the definition of $G(\Os^*) =  \argmin_{\ba( \cdot ) \in \Ep} \J^{\Os^*} (\ba ( \cdot ))$.
Thus, for all $\Es \in \Em$ we have that:
\begin{equation} \label{eq:cond1}
P(\Os^*, \Es^* )  \leq P(\Os^*, \Es) \ .
\end{equation}
From \cref{thm:main2,thm:main3} it follows that for all $\Os \in \Om$:
\begin{equation} \label{eq:cond2}
P(\Os^*, \Es^*) = \sum_{i=1}^r \lambda^*_i \mathbb{E}_{\Es^*}\left[ \J_i( \ba ( \cdot ) )\right] \geq \sum_{i=1}^r \lambda_i \mathbb{E}_{\Es^*}\left[ \J_i( \ba ( \cdot ) )\right] = P(\Os, \Es^*) \ .
\end{equation}
\Cref{eq:cond1,eq:cond2} imply that $(\Os^*, \Es^*)$ is a Nash equilibrium.
\end{proof}

Any mix of $\Os^*$-optimal trajectories forms a $\Os^*$-optimal strategy for the evader.
However, that mix is part of a Nash equilibrium only if the observer has no incentive to change his strategy in response.
\Cref{thm:main}~says that this is the case when the $\Es^*$ defining the mix of individual observability of $\Os^*$-optimal trajectories lies on the central ray of the Pareto Front for a reduced problem. I.e., the PF for the SEG
%O-E game
where the observer has a potentially smaller number of positions (the ones which are used with positive probability in $\Os^*$). This PF is in an $s$ dimensional criterion space, where $s = | \mathcal{I} | \leq r$.
In~\cref{fig:observability1}, the number of observer positions is $r=2$, and the dimension of the ``reduced" problem is also $s=2$ since both positions are used with positive probability.
In this example, a single $\Os^*$-optimal trajectory exists and corresponds to the intersection of the central ray and the convex part of the PF. In the examples from~\cref{fig:NoObstaclesHole} and~\cref{fig:observability2}, we still have $r=2$ and $s=2$, however there are two $\Os^*$-optimal trajectories.
The Nash mixed strategy for E is thus obtained by finding a probability distribution $(\omega_1, \omega_2) \in \Delta_2$ over these two trajectories $(\ba_1 ( \cdot), \ba_2 ( \cdot) )$ such that the linear combination of their individual costs lies on the central ray, i.e., such that $ \omega_1 \J_1 (\ba_1( \cdot) ) + \omega_2  \J_1(\ba_2( \cdot)) =  \omega_1 \J_2 (\ba_1( \cdot) ) + \omega_2 \J_2(\ba_2( \cdot)) $.
In the example from~\cref{fig:observability3}, $r=2$ and $s=1$. The PF of the reduced problem is a single point, and thus trivially lies on the ``central ray", yielding a pure Nash equilibrium strategy for E. In~\cref{sec:examples}, we show additional examples with $r=3$, $s=3$, and $r=6$, $s=4$.
Computationally, \Cref{thm:main} means that if we are able to find a set of $g$ $\Os^*$-optimal control functions $\mathcal{A}(\Os^*) = \{  \ba_j ( \cdot ) \}_{j=1}^{j=g}$, such that~\cref{thm:main2,thm:main3} hold for some probability distribution $\bomega \in \Delta_g$, then $\Os^*$ is O's optimal response to $\bomega$ and we have found a Nash equilibrium pair. Note that the minimum number of trajectories $g$ needed to form a Nash equilibrium is bounded above by $s$.

One remaining task is finding such a set $\mathcal{A}(\Os^*)$.  Multiple $\Os^*$-optimal controls
%TODO: fix "only"?
only exist if $\x_S$ lies on a shockline of $u^{\Os^*}$, where the gradient is undefined (e.g., the $\lim_{\x_i\to\x_S} \nabla u(\x_i)$ can be different depending on the sequence $\{x_i\}_i$).  Numerically, our  approximation of $u^{\Os^*}$ will yield a single upwind approximation of $\nabla u^{\Os^*},$ yielding a single $\Os^*$-optimal trajectory.
As we show in~\cref{fig:shock}, multiple optimal trajectories might lie in the same upwind quadrant and any numerical implementation of gradient descent will find only one of them.  (In theory, one can approximate the other by perturbing $\x_S$, but the direction of perturbation is unobvious, particularly when $\x_S$ lies on an intersection of multiple shocklines, which is surprisingly common in this application as we show in further sections.)
\begin{figure}[tbhp]
\centering
\subfloat[]{\label{fig:shocka}
\includegraphics[width=0.31\textwidth]{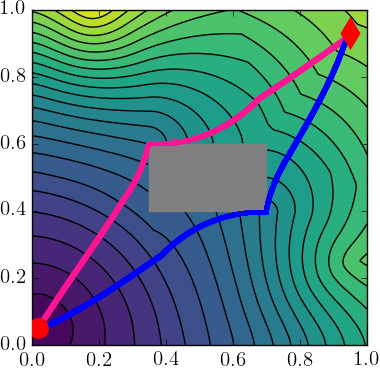}
}
\subfloat[]{\label{fig:shockb}
\includegraphics[width=0.31\textwidth]{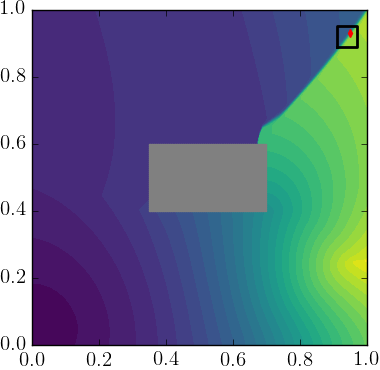}
}
\subfloat[]{\label{fig:shockc}
\includegraphics[width=0.32\textwidth]{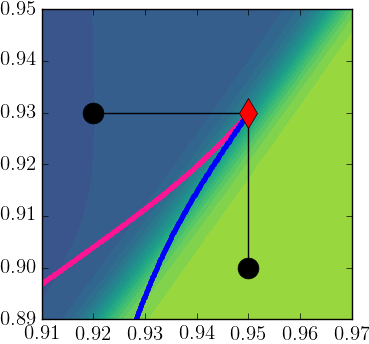}
}
\caption{ \footnotesize
  (a) Two $\Os^*$-optimal trajectories in pink and blue plotted over the level sets of $u^{\Os^*}(x)$. The source location $\x_S$ is on a shockline of $u^{\Os^*}(x)$, the two trajectories have the same expected cumulative observability, but different individual cumulative observability.
    (b) The individual cost function $v_1^{\Os^*}(x)$ is discontinuous at the source $\x_S$. The black square is the region displayed on (c).
    (c) The individual cost function $v_1^{\Os^*}(x)$ zoomed in around the source and a depiction of the upwind stencil.
    The stencil (displayed larger for the sake of visualization) contains a point on either side of the line of discontinuity of $v_1^{\Os^*}(x)$.
}
\label{fig:shock}
\end{figure}

This challenge is even more pronounced because~\ref{alg:supergradient} yields {\em an approximate} value of $\Os^*$, since $\x_S$ will now be only {\em near} a shockline for some perturbed $\Os^*_{\delta} = \Os^* + \delta \Os$.  The resulting single
$\Os^*_{\delta}$-optimal control will be a reasonable approximate solution for the max-min problem, but can be arbitrarily far from the solution to a min-max problem (where O has a chance to switch to another strategy).

In view of these challenges, we opt for a different approach, where an approximation to $\mathcal{A}(\Os^*)$ is computed iteratively, by adaptively growing a collection
of $\Os^*_{\delta}$-optimal controls corresponding to different $\delta \Os$'s.
In some degenerate cases, generating even the first $\ba_1(\cdot) \in \mathcal{A}(\Os^*)$  may not be trivial since some $\Os^*$-optimal control computed by solving the Eikonal will not be necessarily Pareto-optimal.  E.g., in~\cref{fig:observability3} two control functions are $\Os^*$-optimal, but only one of them is used in the Nash strategy of E as the blue trajectory violates~\cref{thm:main3}. However, both trajectories are indistinguishable from the point of view of the Eikonal solver since the position $\hat{\x}_1$ has zero weight in the weighted observability function $\K^{\Os^*}$.
To address this issue whenever $s < r$, we set the weight of the pointwise observability of each unused position $i \not\in \mathcal{I}$ to some small value $\epsilon$ (our implementation uses $ \epsilon = 10^{-6}$). This is equivalent to seeking the solution of the weighted cost Eikonal equation for some perturbed $\Os^*_{\delta} =   (1- \epsilon ) \Os^* + \frac{\epsilon}{r-s} I_{\mathcal{I}^c}$, where $I_{\mathcal{I}^c}$ is the characteristic function of the complement of $\mathcal{I}$.
We now turn our attention to finding further perturbations needed to generate $\Os^*_{\delta}$-optimal trajectories in order to make~\cref{thm:main2} approximately hold. Our goal is to have
\begin{equation} \label{eq:characterization2}
\sum_{j=1}^{g} \omega_j  \J_i (  \ba_j ( \cdot) ) = G( \Os^* )
\end{equation}
approximately hold for all $ i \in \mathcal{I} = \{ i \mid \lambda^*_i > 0 \} $.
Unless this is already true with $g=1$ (based on the previously found $\ba_1(\cdot)$), we will need to find more $\Os^*_{\delta}$-optimal controls.
Without loss of generality assume that $ \mathcal{I} = \{ 1 , \dots, s \}$, and
suppose we have already generated a set of $k$ $\Os^*_{\delta}$-optimal trajectories $\A^{k} = \{\ba_1( \cdot), \ba_2 ( \cdot), \dots ,\ba_{k} ( \cdot ) \}$, for some $k < g$ .
In order for~\cref{eq:characterization2} to approximately hold, we will be increasing $k$ until the norm of residual
\begin{equation}\label{eq:characterizationMatrix}
  \small
  \Res(\bomega) =
   \begin{bmatrix} G(\Os^*) \\ G(\Os^*)\\ \vdots \\ G(\Os^*) \end{bmatrix} -
\begin{bmatrix}
\J_1 (  \ba_1 ( \cdot) ) & \J_1 (  \ba_2 ( \cdot) ) & \hdots & \J_1 (  \ba_{k} ( \cdot) ) \\
\J_2 (  \ba_1 ( \cdot) ) & \J_2 (  \ba_2 ( \cdot) ) & \hdots & \J_2 (  \ba_{k} ( \cdot) ) \\
 \vdots &  \vdots & \ \ddots & \vdots \\
\J_s (  \ba_1 ( \cdot) ) & \J_s (  \ba_2 ( \cdot) ) & \hdots & \J_s (  \ba_{k} ( \cdot) )
\end{bmatrix}
\begin{bmatrix} \omega_1 \\\omega_2 \\ \vdots \\ \omega_{k} \end{bmatrix}
\end{equation}
falls under a threshold $\text{tol}_{\Res}$. Assuming the set of trajectories $\A^k$ has already been computed, the probability distribution $\bomega^k \in \Delta_{k}$ minimizing the norm of this residual $\|\Res(\bomega^k)\|_2 $ can be found by quadratic programming.
The residual vector $\Res(\bomega^k)$ provides information about which control functions are missing. For example, consider the case where we observe that a single entry of $\Res(\bomega^k)$ is large and positive, i.e., that for some $i \in \mathcal{I}$:
\[
\sum_{j=1}^{k} \omega_j^k  \J_i (  \ba_j ( \cdot) ) << G( \Os^* ) \ .
\]
The characterization in~\cref{thm:main} implies that $\A ( \Os^*)$ should include at least one trajectory much more observable from position $\hat{\x}_i$.
A $\Os^*_{\delta}$-optimal trajectory with that property can be found by perturbing $\Os$ to slightly decrease the role of $\xhat_i$ in O's chosen strategy.
This is equivalent to re-solving the Eikonal with $\K^{\Os^*_{\delta}}$ corresponding to $\Os^*_{\delta} = \Pi_{\mathcal{I}} \left( \Os^* - \delta \e_i \right)$
 where $\delta << 1$ is chosen adaptively (see~\cref{alg:param}), $\e_i$ is the $i$-th standard basis vector, and $\Pi_{\mathcal{I}}$ is the orthogonal projection onto the simplex defined only with elements of $\mathcal{I}$.
Once a new $\Os^*_{\delta}$-optimal control function has been found, we may solve the quadratic program in~\cref{eq:characterizationMatrix} again with an additional column, and repeat the process until the norm of the residual is sufficiently small.
More generally, a large $\|\Res(\bomega)\|$ implies that some control functions in $\Ep (\Os^*) $ (or some mix of control functions) not in the current set $\A^k$ has a high observability with respect to the positive entries of $\Res(\bomega)$
while having a low observability with respect to the negative entries of $\Res(\bomega)$. Thus, we set the perturbation direction to $-\Res(\bomega)$ instead of $-\e_i$. Throughout this perturbation step, the entries of $\Os^*$ associated with the complement $\mathcal{I}$ are held fixed. Our full method for computing an approximate Nash equilibrium is summarized in~\cref{alg:mainalg}.

This method also has a geometric interpretation in terms of the Pareto Front. Whenever $\Es^*$ is not a pure strategy, a hyperplane normal to $\Os^*$ supports PF at multiple points (corresponding to all controls in $\A ( \Os^*)$). However, any generic perturbation of $\Os^*$ would result in a hyperplane supporting PF near only one of these points, and the approximation to $\Os^*$ found by~\cref{alg:supergradient}, will correspond to a single optimal trajectory.
For example, if we start with $\ba_1 ( \cdot)$ corresponding to the yellow point in~\cref{fig:obs2c} (and associated yellow trajectory in~\cref{fig:obs2b}), then a small tilt (decreasing the role of position $\hat{\x}_1$ in O's plan) will yield a hyperplane supporting PF near the green point, allowing us to approximate the green trajectory in~\cref{fig:obs2b} by solving the weighted cost Eikonal equation with observability function $\K^{\Os^*_{\delta}}$.
\begin{algorithm}[h]
\begin{algorithmic}[1]
  \caption{ Computing an approximate Nash equilibrium of the SEG. }\label{alg:mainalg}
\STATE{Find $\Os^*$ using~\cref{alg:supergradient} }
\STATE{ $\mathcal{I} \gets \{i \mid \lambda^*_i > \text{tol}_{\lambda}  \}$ }
\STATE{    $\Os^*_{\delta} \leftarrow (1- \epsilon ) \Os^* + \frac{\epsilon}{r-s} I_{\mathcal{I}^c} $ }
\STATE{ Find $\Os^*_{\delta}$-optimal control $\ba_{1} ( \cdot )$ and compute $\J_i( \ba_{1}( \cdot ))$ for all $i \in \mathcal{I}$}
\STATE{ $k \gets 1$, $\A^1 \gets \{ \ba_1 ( \cdot) \} $, $\bomega^1  \gets 1 $}
\WHILE{ $\| \Res(\bomega^k)\| >  \text{tol}_{\Res} $}
\STATE{    $\Os^*_{\delta} \leftarrow (1- \epsilon ) \Pi_{\mathcal{I}} \left(\Os^* - \delta \Res(\bomega^k) \right) + \frac{\epsilon}{r-s} I_{\mathcal{I}^c} $ }
\STATE\label{line8}{ Find $\Os^*_{\delta}$-optimal control $\ba_{k+1} ( \cdot )$ and compute $\J_i( \ba_{k+1}( \cdot ))$ for all $i \in \mathcal{I}$ }
\STATE{ $\A^{k+1} \gets \A^{k} \cup \{ \ba_k ( \cdot )  \}$}
\STATE{ $\bomega^{k+1} \leftarrow  \argmin_{ \bomega^{k+1} \in \Delta_{k+1} } \| \Res(\bomega^{k+1}) \|_2 $}
\STATE{ $k \gets k + 1$}
\ENDWHILE\\
\RETURN $\Os^*$, $\A^k$, $\bomega^k$
\end{algorithmic}
\end{algorithm}

%% file: Examples.tex
% !TEX root = draft.tex

\section{Numerical matters}\label{sec:examples}
In this section, we detail the implementation of our algorithm and present additional numerical results.
All algorithms were implemented in C++ and compiled
with icpc version 16.0 on a MacBook Pro (16 GB RAM and an Intel Core i7 processor with
four 2.5 GHz cores). The code is available online at
\url{https://github.com/eikonal-equation/Stationary_SEG}.
Our implementation relies on data structures and methods from Boost, Eigen and QuadProg++ libraries.

\subsection{Functions, parameters, methods}\label{sec:parameters}
All of our examples are posed on the domain $\domain = [0,1]^2$ with the possible exclusion of obstacles.
All figures are based on computations on a uniform cartesian grid of size  $ n \times n = 501 \times 501$ (with the grid spacing $h=1/500$).
To simplify the discussion, we
always use a constant speed function $f(x) = 1$ though any inhomogeneous speed can be similarly handled by solving the Eikonal equation~\cref{eq:Eikonal}.

The pointwise observability functions are defined as
\[
\K_i(\x) =
\begin{cases}
   \sigma, & \text{if $\x$ is in a shadow zone of $\hat{\x}_i$}; \\
   \hat{\K} ( | \x - \hat{\x}_i | ) \, + \, \sigma, & \text{otherwise}.
\end{cases}
\]
We set $\sigma = 0.1$ and $\hat{\K}(r) = (\rho r^2 + 0.1 )^{-1}$ with $\rho = 1$ in all examples except in~\cref{fig:observability3} (where we set $\rho = 30$ simply to improve the visualization).  The visibility of each gridpoint with respect to each observer position is precomputed and stored, but the $\K_i$ values are computed on the fly as needed.

The shadow zones for each observer are precomputed as follows.
For each observer location $\hat{\x}_i$, two distance functions are computed:  $D_0^i(\x)$ and $D^i(\x)$. The first is the distance between $\hat{\x}_i$ and $\x$ when the obstacles are absent, while the second is that distance when obstacles are present.
These distance functions can be computed by imposing the boundary conditions $D_0^i(\hat{\x}_i)=D^i(\hat{\x}_i)=0$ and then solving two Eikonal equations~\cite{SethBook2}:
\begin{equation}
\big| \nabla D_0^i(\x) \big| = 1, \qquad  \big| \nabla D^i(\x) \big| = \text{Obs}(\x),
\end{equation}
with $\text{Obs}(\x)$ set to $\infty$ inside the obstacles and $1$ otherwise.
The shadow zone of  $\hat{\x}_i$ is characterized by $D^i > D_0^i$. But due to numerical errors in their approximation,
we use a threshold value  $\tau = 10^{-3}h$ (where $h$ is the grid spacing)
and specify that $\x$ is in this shadow zone whenever $D^i(\x) > D_0^i(\x) + \tau.$

The perturbation stepsize $\delta$ in~\cref{alg:mainalg} is chosen adaptively using~\cref{alg:param}. The goal of the adaptive strategy is to find the smallest perturbation $\delta$ necessary to obtain an additional $\Os^*_{\delta}$-optimal control function $\ba_{k+1} ( \cdot)$.
\begin{algorithm}[H]
\begin{algorithmic}[1]
  \caption{ Adaptive strategy for choosing $\delta$ to generate $\ba_{k+1}( \cdot) $}\label{alg:param}
\STATE{ $\delta  \gets \delta_0$  }
\STATE{ $\Os^*_{\delta} \leftarrow (1- \epsilon ) \Pi_{\mathcal{I}} \left(\Os^* - \delta \Res(\bomega^k) \right) + \frac{\epsilon}{r-s} I_{\mathcal{I}^c} $ }
\STATE{ Compute a $\Os^*_{\delta}$-optimal control function $\hat{\ba} ( \cdot )$ }
\WHILE{ $\| \bJ(\hat{\ba} ( \cdot)) - \bJ(\ba_{j} ( \cdot ) )\|_{2} < \text{tol}_{\delta}$  for any $j \in \{ 1, \dots k \}$ }
\STATE{ $\delta \gets 2 \delta$  }
\STATE{ $\Os^*_{\delta} \leftarrow (1- \epsilon ) \Pi_{\mathcal{I}} \left(\Os^* - \delta \Res(\bomega^k) \right) + \frac{\epsilon}{r-s} I_{\mathcal{I}^c} $ }
\STATE{ Compute $\Os^*_{\delta}$-optimal control function $\hat{\ba} ( \cdot )$ }
\ENDWHILE\\
\STATE{ $\ba_{k+1}( \cdot ) \gets \hat{\ba}( \cdot ) $}
\end{algorithmic}
\end{algorithm}
The initialization used in our implementation is $\delta_0  =10^{-4}$, and the tolerance is set to $\text{tol}_{\delta} = 10^{-2} \| \bJ(\hat{\ba} ( \cdot)) \|_{2} $.
The stepsize rule used in the supergradient iteration in~\cref{alg:supergradient} is $ \alpha_k = 1 \big/ ( k \| \bJ( \ba^{\Os_0} ( \cdot) ) \| )$, the initial guess $\Os_0$ is a uniform distribution on $\Ep$ and the tolerance criteria on the residual and the near $0$ entries used in~\cref{alg:mainalg} are $\text{tol}_{\Res}= 10^{-3} G(\Os^*)$ and $\text{tol}_{\lambda} = 5 \cdot 10^{-3}$ respectively.
The quadratic programming problem in~\cref{eq:characterizationMatrix} is solved using the library QuadProg++.

\subsection{Computation of individual costs} \label{sec:computingcosts}
Running~\cref{alg:supergradient} requires computing the vector of individual observability $\bJ ( \x_S, \ba^{\Os} ( \cdot ) )  $.
This problem is exactly the one solved by the scalarization approach described
 in~\cref{sec:MultiObjective}. Therefore, it can in principle be done by solving the Eikonal equation in~\cref{eq:Eikonal} with cost function $\K^{\Os}$
and associated linear equations in~\cref{eq:associatedlinear};
i.e.:  $G(\Os) = u ^{\Os} ( \x_S) $ and $\J_i ( \x_S, \ba^{\Os} ( \cdot ) ) = v_i ^{\Os} (\x_S) $. However, this technique has a severe drawback for this particular application: at the optimal $\Os^*$, $v^{\Os^*}_i$ is often discontinuous at $\x_S $. E.g., in~\cref{fig:shockb}, the
upwind stencil containing the two $\Os^*$-optimal trajectories contains a point on either side of the discontinuity line of $v_1^{\Os^*}$ (which is the shockline of $u^{\Os^*}$). As a result, the value of $v_1^{\Os^*}( \x_S)$ is updated by interpolating the discontinuous function $v_1^{\Os^*}$ across the line of discontinuity.

This effect happens when multiple trajectories are $\Os^*$-optimal.
Each of these trajectories has the same expected cumulative observability $\J^{\Os^*} = \sum_{i} \lambda^*_i \J_i$, but different individual observability $\J_i$.
This issue leads to a large numerical error when using $v^{\Os}_i(\x_S)$ to estimate the supergradient in~\cref{alg:supergradient}, causing poor convergence of the method. Instead, we use the following process to compute the individual costs: first we solve the weighted cost Eikonal equation~\cref{eq:WeightedCostEikonal} to obtain $u^{\Os}$ for a fixed $\Os$, then we trace the path $\y(t)$ using a gradient descent method on the value function $u^{\Os}$ and numerically estimate the integrals:
\[
\mathcal{J}_i( \x_S , \ba^{\Os} ( \cdot )) = \int_{0}^{T_{\ba^{\Os}}} \K_i( \y (t), \ba^{\Os}(t) )\, dt, \qquad i=1,\ldots, r.
\]

\subsection{Additional experiments and error metrics}
We present two additional examples that include a higher number of observer plans.
In~\cref{fig:3paths}, we show an example where the mixed strategy Nash equilibrium consists of a distribution over three strategies for both the evader and the observer.
\Cref{fig:3paths} shows the value function $u^{\Os^*}$ at the optimal $\Os^*$. We observe that three shocklines of the value function $u^{\Os^*}$ meet at the source location $\x_S$, which implies that four trajectories are optimal starting from this location. However, the minimax theorem for infinite games assures that only 3 pure strategies are necessary to form a Nash equilibrium.
Using~\cref{alg:mainalg}, we find an approximate Nash equilibrium which uses a mix of such three trajectories.
\
\begin{figure}[h]
\begin{center}
\includegraphics[width=0.50\linewidth,keepaspectratio]{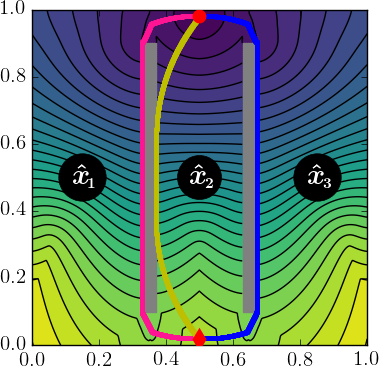}
\end{center}
\caption{ Computed Nash equilibrium for a situation where a mix of three pure strategies are necessary for each player. The value function $u^{\Os^*}$ with three near-$\Os^*$-optimal trajectories in pink, blue and yellow. Part of the pink is obstructed by the blue and green path. The optimal strategy for O is $\Os^* = \left[ p(\x_1), p(\x_2), p(\x_3)\right] = \left[ 0.34, 0.32, 0.34 \right] $, and the optimal strategy for E consists of three trajectories used with probability $\bomega^* = \left[ p(\text{blue}),p(\text{yellow}),p(\text{pink})  \right] = \left[ 0.40, 0.20, 0.40 \right] $.
In this example, the pink and yellow $\Os^*$-optimal trajectories initially coincide near $\x_S$, hence one cannot find both of them by perturbing the initial position $\x_S$.
}
\label{fig:3paths}
\end{figure}

In Figure~\ref{fig:Labyrinth}, we show a maze-like example where the observer may choose among six possible positions. Using~\cref{alg:mainalg}, we determine that at the approximate Nash equilibrium, only four positions are used with positive probability by O, and E uses four different trajectories which are displayed in~\cref{fig:Labyrinth}.

\begin{figure}[h]
\begin{center}
{\includegraphics[width=0.5\linewidth,keepaspectratio]{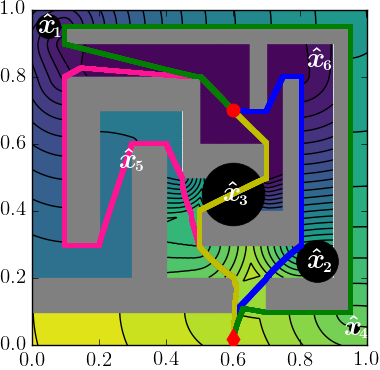}}
\end{center}
\caption{ Computed Nash equilibrium for a maze-like example. The value function $u^{\Os^*}$ and four near-$\Os^*$-optimal trajectories in pink, blue, yellow and green. The approximate Nash equilibrium strategy for O is $\Os^* = \left[ p(\hat{x}_i)\right]_{i=1}^{i=6} = \left[ 0.174,0.301,0.452,0.073,0,0\right] $. The approximate Nash equilibrium strategy for E uses four trajectories with probability $\bomega^* = \left[ p(\text{pink}),p(\text{yellow}),p(\text{blue}),p(\text{green})  \right] = \left[ 0.246,  0.461,  0.144,  0.149 \right] $.
}\label{fig:Labyrinth}
\end{figure}

In order to test the performance of~\cref{alg:mainalg}, we consider three error metrics:
\begin{enumerate}
  \item \textit{The optimization error} in $G(\Os)$ arises from several effects: the discretization error of the Eikonal solver, the discretization error of the path tracing and path integral evaluation, and the early stopping of the supergradient iterations. To generate the ``ground truth", we performed the same computation on a finer grid of size of $n = 2001 \times 2001$ (i.e. we consider a grid with 16 times more unknowns) and run the supergradient iteration until we observe stagnation in the objective function value of the iterates. We approximate the relative error in our computations on a $501\times 501$ grid as:
  \[
  E_{rel} \left[ G(\Os^*) \right]  =
  \left| G_{501}(\Os_{501}^*) - G_{2001}(\Os_{2001}^*) \right| \, \big/ \, G_{501}(\Os_{501}^*)  \ .
  \]

  \item \textit{The Observer's regret} estimates how much the observer could improve his payoff by unilaterally deviating from our approximate Nash equilibrium. (Recall that, if the approximate Nash equilibrium were exact, the observer would not be able to increase his payoff at all). We quantify this error using the  normalized residual in~\cref{eq:characterizationMatrix}, i.e.:
\[
\text{Observer's regret} = \left\| R(\bomega) \right\|_2 \, \big/ \, \left( |\mathcal{I}| G(\Os^*) \right) \ .
\]
  \item \textit{The Evader's regret} estimates how much the evader could improve his payoff by unilaterally deviating from our approximate Nash equilibrium. This corresponds to how far from $\Os^*$-optimal are the controls produced by~\cref{alg:mainalg}. Recall that the control function $\ba_1 (\cdot)$ is (up to numerical errors) $\Os^*$-optimal, whereas $\ba_{k}(\cdot)$ for $k \geq 2$ are $(\Os^* + \delta \Os)$-optimal. We report the maximum relative error in $\Os^*$ cumulative observability of the $(\Os^* + \delta \Os)$-optimal trajectories, that is:
  \[
  \text{Evader's regret} =
\max_{k} \left| \J^{\Os^*} (  \ba_1 ( \cdot )) - \J^{\Os^*} (  \ba_k ( \cdot )) \right| \, \big/ \, \J^{\Os^*} (  \ba_1 ( \cdot ))
  \]
\end{enumerate}

These error metrics are reported in ~\cref{tab:results} along with timing metrics for each example presented in the paper.
\begin{table}[tbhp]
{\footnotesize
  \caption{Table of timing and error metrics. The error metrics are described in the main body of the text.  }\label{tab:results}
\begin{center}
  \begin{tabular}{r  c  c  c c  c  }
  \toprule
    & \cref{fig:observability1} & \cref{fig:observability3} & \cref{fig:observability2} & \cref{fig:3paths}&  \cref{fig:Labyrinth}\\ [0.5ex]
  \midrule
  Number of it. of~\cref{alg:supergradient} & 100 &  100 & 100 & 300 & 400 \\
  Total CPU time (seconds) &61 & 61 & 69 & 198 &321 \\
  $  E_{rel} \left[ G(\Os^*) \right]$ & $1 \cdot 10^{-3} $ & $1 \cdot 10^{-3} $ & $9 \cdot 10^{-4} $ &  $1 \cdot 10^{-3} $ &$3 \cdot 10^{-4} $  \\
  Observer's regret&  $1 \cdot 10^{-4} $ & $0$ & $3 \cdot 10^{-4} $ & $4 \cdot 10^{-6} $ &$1 \cdot 10^{-4}$  \\
  Evader's regret & $0 $ &$0$  & $2 \cdot 10^{-3} $ & $2 \cdot 10^{-3} $ & $2 \cdot 10^{-2}$  \\
  \bottomrule
  \end{tabular}
\end{center}
}
\end{table}

%% file: Extensions.tex
% !TEX root = draft.tex

\section{Extension to groups of evaders}
\label{sec:MultipleEvaders}
We now consider an extension of the surveillance-evasion game to a game which involves a team of $q$ evaders. Each evader $\text{E}^l$ chooses a trajectory leading him from his own source location $\x^l_S$ to a target location $\x^l_T$, according to his own speed function $f^l(x)$.
The pointwise observability function $\K^{\Os}$ is shared for all evaders and depends only on the strategy $\Os$ of the observer.
This induces $q$ different cumulative observability functions $\J^{l,\Os} ( \x_S^l, \ba^l( \cdot) )$ defined as in~\cref{eq:IntegratedObservability}, and $q$ different value functions $u^{l,\Os}$ which are solutions of Eikonal equations with $q$ different boundary conditions.

In this version of the game, we assume that a central organizer for evaders faces off against the observer. The goal of that central organizer is to minimize the weighted sum of evaders' cumulative expected observabilities. The weights $\{ w_l\}_{l=1}^{l=q}$ in the sum reflect the relative importance of each evader.  We further assume that the central organizer and the observer agree on that relative importance, making this a two player zero-sum game with a payoff function defined by:
\begin{equation} \label{eq:multEpayoff}
P( \Os, \{\Es^l\}_{l=1}^q ) = \sum_{l=1}^{l=q} w_l  \mathbb{E}_{\Es^l } \left[ \J^{l,\Os} ( \x_S^l, \ba^l( \cdot ) )\right] \ .
\end{equation}

Although we focus on a zero-sum two player game, we note that its Nash equilibrium $\left( \Os^*, \{ \Es^l \}_{l=1}^{l=q} \right)$ must also be among Nash equilibria of a different $(q+1)$-player game: the one, where each of the $q$ evaders is selfishly minimizing their own cumulative observability $\J^{l,\Os} ( \x_S^l, \ba_l( \cdot) )$, while the observer still attempts to maximize the crowd-wide observability in~\cref{eq:multEpayoff}.
This property
follows from two simple facts:
\begin{enumerate}
\item
The Observer's payoff is the same in both versions of the game and thus cannot be improved unilaterally in a $(q+1)$ player game.
\item
In the Nash equilibrium for the two-player game, the central organizer would only ask each evader to assign positive probabilities to their $\Os^*$-optimal trajectories. (Otherwise, the weighted sum in~\cref{eq:multEpayoff} could be improved). Thus, they would also be maximizing their individual payoffs.
\end{enumerate}

In this new setting,~\cref{thm:main} holds and the observer's half of the Nash equilibrium may be found by maximizing the concave function:
\begin{equation} \label{eq:WeightedSumObs}
G^q(\Os) = \min_{\ba^l ( \cdot )} \sum_{l=1}^{l=q} w_l \J^{\Os} ( \x_S^l, \ba^l( \cdot) ) \ .
\end{equation}
The function $G^q(\Os)$ and its supergradients may be evaluated in a similar way to~\cref{sec:computingcosts},
but require $q$ solves of the Eikonal equation with different boundary conditions and speed functions, and the numerical evaluation of $q\times r$ path integrals.
However, we note that if all evaders have the same speed function and share the same target location (or, alternatively, share the same source location), only a single Eikonal equation solve is in fact required.
With minor modifications,~\cref{alg:mainalg} may be also applied to solve this version of the problem.
For each perturbation of $\Os^*$ a set of $q$ control functions is generated on line~\ref{line8} of~\cref{alg:mainalg}, with one control function found for each evader. Although we obtain a new set of $q$ control functions for each perturbation, some of the control functions for specific evaders may be essentially the same as those already obtained from previous perturbations. We address this in post-processing, by pruning the output of modified~\cref{alg:mainalg} to identify distinct trajectories for each evader.

We show numerical results for two test problems with $q=2$ equally important evaders (i.e., $w_1=w_2$) in each of them.
An example presented in~\cref{fig:2evaders} uses the same obstacle and the same $r=2$ possible observer locations already used in~\cref{fig:observability2}.
At the approximate Nash equilibrium found using~\cref{alg:mainalg}, the observer uses these two locations with probabilities $\Os^* = (0.35,0.65)$ and the central controller directs both evaders to use pure policies: deterministically choose pink and blue trajectories to their respective targets.  Even though the first evader's starting position and destination are also the same as in~\cref{fig:observability2}, his (and the Observer's) optimal strategies are quite different here due to the second evader's participation.

\begin{figure}[tbhp]
\centering
\subfloat[]{\label{fig:2evadersa}
\includegraphics[width=0.48\textwidth]{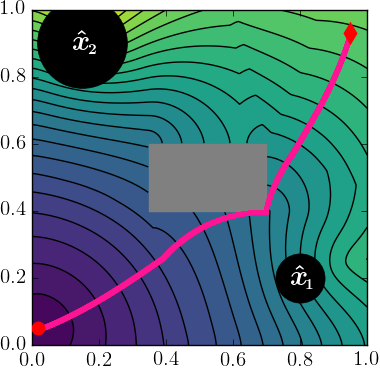}
 }
 \subfloat[]{\label{fig:2evadersb}
 \includegraphics[width=0.48\textwidth]{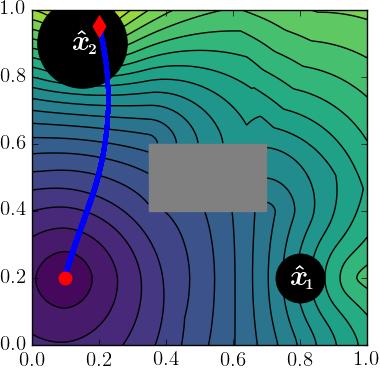}
  }
\caption{ \footnotesize
 Computed approximate Nash equilibrium for a group of two evaders. The approximate Nash equilibrium pair of strategies is $\Os^*$ for O, and a single $\Os^*$-optimal trajectory for each evader.  (a) The value function $u^{1,\Os^*}$ for $ \Os^* = \left[ 0.35,0.65 \right]$ of evader $1$, and the $\Os^*$-optimal trajectories for evader 1 shown in pink. (b) The value function $u^{2,\Os^*}$ for the same $\Os^*$ of evader $2$, and his $\Os^*$-optimal trajectory shown in blue.
}
\label{fig:2evaders}
\end{figure}
In a maze-like example presented in~\cref{fig:Labyrinth2evaders}, O can choose among six possible locations, but his optimal mixed strategy $\Os^*$ uses only four of them.
\cref{alg:mainalg} yields three sets of two near-$\Os^*$-optimal trajectories which form an approximate Nash equilibrium,
but they only contain two distinct trajectories for each of the evaders.
We report timing and error metrics for these two examples in~\cref{tab:resultsExtension}.

\begin{figure}[tbhp]
\centering
\subfloat[]{\label{fig:Labyrinth2evadersa}
\includegraphics[width=0.48\textwidth]{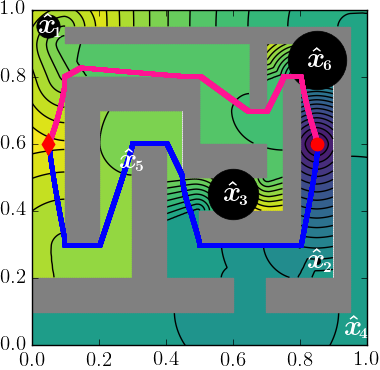}
 }
\subfloat[]{\label{fig:Labyrinth2evadersb}
\includegraphics[width=0.48\textwidth]{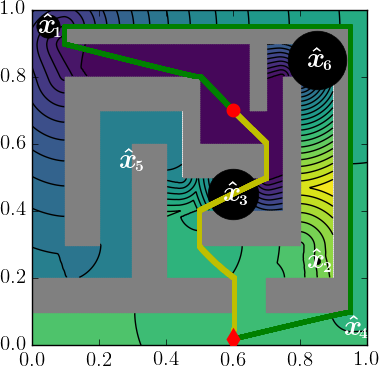} }
\caption{ \footnotesize
 Computed approximate Nash equilibrium for a maze-like example with two evaders.    (a) The value function $u^{1,\Os^*}$ of evader $1$, and two near-$\Os^*$-optimal trajectories for evader plotted in pink and blue.
 (b) The value function $u^{2,\Os^*}$ of evader $2$, and two near-$\Os^*$-optimal trajectories for evader 2 plotted in yellow and green.
 The approximate Nash equilibrium $(\Os^*, \Es^*)$ is $\Os^* = \left[ p ( \hat{x}_i) \right] =\left[ 0.168, 0.0455, 0.364, 0, 0,0.422 \right]$,
 and $\Es^*$ consists of a mixed strategy for the group of evaders.
 The mixed strategy of evader 1 is
 $\left[ p(\text{pink}),p(\text{blue}) \right] = \left[0.85, 0.15\right]$,
  and the mixed strategy for evader 2 is
  $ \left[ p(\text{yellow}) , p(\text{green}) \right] = \left[ 0.89, 0.11 \right].$
}
\label{fig:Labyrinth2evaders}
\end{figure}

\begin{table}[tbhp]
{\footnotesize
  \caption{Table of running times and errors for examples with multiple evaders. }\label{tab:resultsExtension}
\begin{center}
  \begin{tabular}{ r  c  c}
  \toprule
    & \cref{fig:2evaders} &\cref{fig:Labyrinth2evaders} \\ [0.5ex]
  \midrule
  Number of it. of~\cref{alg:supergradient} & 353 &  300  \\
  Total CPU time (seconds) & 631 & 594  \\
  $  E_{rel} \left[ G(\Os^*) \right]$ &  $5 \cdot 10^{-4}$ &  $7 \cdot 10^{-3}$ \\

  Observer's regret&  $5 \cdot 10^{-4}$  & $ 1 \cdot 10^{-3}$   \\
  Evader's regret & $5 \cdot 10^{-3} $ & $1 \cdot 10^{-2}$ \\
  \bottomrule
  \end{tabular}
\end{center}
}
\end{table}

%% file: Conclusion.tex
% !TEX root = draft.tex

\section{Conclusion}
\label{sec:Conclude}

We have considered an adversarial path planning problem, where the goal is to minimize the cumulative exposure/observability to a hostile observer.  The current position of the latter is unknown, but the full list of possible positions is assumed to be available in advance. The key assumption of our model is that neither the Evader (E) nor the enemy Observer (O) can adjust their plan in real time based on the opponent's state and actions.  Instead, both of them are required to choose their (possibly randomized) strategies in advance.  We discussed two versions of this problem; in the first one, a completely risk-averse evader attempts to minimize his worst-case cumulative observability.
We showed that this version can
be solved using previously developed methods for multiobjective path planning.
However, the solution %of this version of the game
is prohibitively computationally expensive when O has a large number of surveillance plans to choose from.
In the second version, the subject of optimization is the E's expected cumulative observability on its way to the target.
We modeled this as a zero-sum Surveillance-Evasion Game (SEG) between two players: E (the minimizer) and O (the maximizer).
We then presented an algorithm combining ideas from continuous optimal control, the scalarization approach for multiobjective optimization, and convex optimization which allows us to quickly compute an approximate Nash equilibrium of this semi-infinite strategic game.
Finally, we showed that this algorithm extends to solve a similar problem involving a group of multiple evaders controlled by a central planner.
The presented algorithm displays at most linear scaling in the number of observation plans, but further speed up techniques would be desirable; the computational bottleneck (numerically solving the Eikonal equation) could be alleviated with domain restriction methods~\cite{clawson2014causal} and factoring approaches~\cite{qi2018corner}.

Although this paper focused on isotropic problems, the anisotropic observer case could be treated in a similar fashion.  (In practice, the pointwise observability might depend on the angle between the evader's direction of motion and the observer's line of sight.) This generalization will have to rely on fast numerical methods developed for anisotropic HJB PDEs; e.g.,~\cite{SethVlad3,alton2012ordered,mirebeau2014efficient,tsai2003fast}.
In a follow-up paper \cite{TD_SEG}, we show that
time-dependent observation plans (e.g., different patrol routes) can be similarly treated by solving $\lambda$-parametrized finite-horizon optimal control problems with numerical methods for time-dependent HJB equations; e.g.,~\cite{Falcone_book,shu2007high}.

We note that the computational cost of our algorithm increases quickly with the number of evaders considered.
The case involving a large number of selfish evaders could be covered by considering the evolution of a time-dependent density of observers, and treating the problem using mean field games~\cite{gueant2011mean,carmona2017probabilistic}.
Another possible extension would be to consider a group of observers choosing among a larger set of surveillance plans.
In that situation, the set of pure strategies of the observers could increase exponentially,
but we anticipate that the computational cost will grow much slower since the number of required Eikonal solves would not increase.